\documentclass[12pt]{amsart}
\usepackage{mystyle}

\usepackage[T1]{fontenc}   

\begin{document}

\renewcommand{\bf}{\bfseries}
\renewcommand{\sc}{\scshape}
\vspace{0.5in}

\title[Root laminations of arbitrary degree]
{Root laminations of arbitrary degree}

\author[A.~Blokh]{Alexander~Blokh}

\thanks{The first named author was partially
supported by NSF grant DMS--2349942}

\thanks{The second named author was partially
supported by NSF grant DMS--1807558}

\author[L.~Oversteegen]{Lex Oversteegen}

\author[V.~Timorin]{Vladlen~Timorin}

\thanks{The third named author was supported by
the HSE University Basic Research Program.
}

\maketitle
\begin{abstract}
This paper studies the space of degree $d>1$ invariant q-laminations, i.e.,
 geodesic laminations invariant under the $d$-tupling map of the circle and
 associated with equivalence relations.
Our main construction associates a q-lamination with any degree $d$ critical portrait \emph{in a canonical way}.
Even though somewhat technical, this is the initial step in the program
 of classification of all degree $d$ invariant q-laminations.
\end{abstract}

\section{Introduction}\label{s:intro-new}
Invariant geodesic laminations play important roles in low-dimensional geometry and topology.
The spaces of laminations invariant under a Fuchsian group and equipped with certain transverse measures
 (1) compactify Teichm\"uller spaces of surfaces,  (2) represent pseudo-Anosov dynamics in the
 classification theorem of surface diffeomorphisms, (3) describe geometry of ends of hyperbolic 3-manifolds (via ending laminations),
 and (5) bridge hyperbolic geometry, topology, and measured foliations.
A similar but less studied version of invariant laminations appears in complex polynomial dynamics \cite{thu85}.
Here, invariance is understood with respect to the $d$-tupling map $\si_d$ of the circle (see Sections \ref{s:lami-gen} and \ref{ss:sib-prop}
 for definitions), where $d>1$ is the degree of polynomials under consideration.

Given a single variable degree $d>1$ complex polynomial $P:\C\to\C$ with a locally connected Julia set $J_P$,
 one associates an invariant lamination $\lam_{\sim P}$ with $P$, which can be understood in two equally important ways.
On the one hand, $\lam_{\sim_P}$ visualizes a certain equivalence relation on the unit circle $\uc$
 so that the quotient of $\uc$ under this equivalence relation is homeomorphic to $J_P$.
Moreover, the homeomorphism can be chosen to topologically conjugate the quotient dynamics of $\si_d$ on $\uc/\sim_P$
 with that of $P:J_P\to J_P$.
Laminations representing equivalence relations are called \emph{q-laminations}
 (or clean laminations, in the language of Thurston \cite{thu85}).
There are ways of associating q-laminations with more general polynomials $P$,
 where $J_P$ is connected but not necessarily locally connected, however, given $P$, an associated lamination may fail to be unique.
Classification of q-laminations is important for understanding of the topology of the
 polynomial connectedness locus.

On the other hand, invariant geodesic laminations are interesting geometric objects on their own right,
 even if they fail to represent equivalence relations.
General invariant laminations are studied for the following reasons:
 (1) they are easier to construct (it suffices to take finitely many chords and add all their iterated pullbacks and limits thereof),
 (2) the space they form is compact,
 (3) subspaces of the latter space give rise to combinatorial models for certain subspaces of complex polynomials.
Still, in this paper, we are mostly interested in q-laminations.
An attempted classification of all q-laminations must involve, as the first step,
 a way to label q-laminations with more tractable combinatorial objects.
Such objects have been known since \cite{thu85}; these are \emph{critical portraits}
 (combinatorial objects with a very similar role are called majors in \cite{tby20}).

A chord $\ol{ab}$ of the unit circle $\uc$ with endpoints $a$, $b\in\uc$ is said to be
 ($\si_d$)-\emph{critical} if $\si_d(a)=\si_d(b)$.
By a ($\si_d$-)\emph{critical portrait}, we mean a collection of $d-1$ critical chords that do not form cycles.
Having fixed a critical portrait allows one to take pullbacks of chords in an almost unambiguous fashion,
 therefore, it allows one to construct invariant laminations.
Still, there may be more than one invariant lamination associated with the same critical portrait,
 and these laminations may fail to be q-laminations.
Thus, the problem is to associate a q-lamination with every critical portrait \emph{in a canonical way}.
We produce such an association, and this is the main result of this paper.
Even though this result is somewhat technical, it is an important initial step in the program
of classification of all q-laminations.

As a tool, we use \emph{proper} laminations. An invariant lamination $\lam$ is proper if and only if
it has no leaves with one endpoint periodic of period $n$ and the other endpoint not periodic of period $n$
(proper laminations were defined in \cite{bmov13} differently; we prove
in Lemma \ref{l:crit-equiv} that the new definition is equivalent to the original one). Now, any
lamination $\lam$ defines an equivalence relation on $\uc$ by declaring two points equivalent if they are the endpoints of a leaf of $\lam$
and then extending this by transitivity. This gives rise to an equivalence relation denoted by $\approx_\lam$.
In Lemma \ref{l:crit-equiv} we also prove that, for a proper lamination $\lam$, $\approx_\lam$ is a laminational equivalence relation.
Then we say that a proper lamination $\lam$ \emph{generates} the q-lamination $\lam_{\approx_\lam}$.

Say that a q-lamination $\lam_\sim$ associated with an equivalence relation $\sim$
is a \emph{root lamination} of a critical portrait $\Pc$ if (1) whenever a component $C$ of $\bigcup\Pc$
does not contain $\si_d$-periodic points, all points of $C\cap \uc$ are $\sim$-equivalent, and
(2) there exists a proper lamination $\lam'$ which generates $\lam_\sim$ such that
if a component $C$ of $\bigcup\Pc$ contains a periodic point $a$ of $\si_d$, then
$C$ lies in the closure $U$ of a complementary component of $\lam'$
so that the period of $U\cap\uc$ is the same as the period of $a$, and the
topological degree of $\si_d:U\cap\uc\to \si_d(U\cap\uc)$
is the same as the number of critical chords in $C$.

\begin{thmain}
Each critical portrait corresponds to a unique root lamination.
\end{thmain}

Existence is proved in Theorem \ref{t:rootexist}, and the uniqueness is proved in Theorem \ref{t:rootuni}.

\noindent\textbf{Acknowledgments.} We are indebted to Professor Yan Gao for drawing our attention to an inaccuracy
in the first version of the paper.

\section{Previous work and notation}\label{s:intro}

We assume familiarity with complex polynomial dynamics and its notation
($K_f$ is the filled Julia set of a polynomial $f$, $J_f$ is the Julia set of $f$, etc) as well
as with some classical notation of various concepts. The boundary of $X\subset \C$ is denoted by $\bd(X)$.

Let $\poly_d$ be the space of all \emph{monic centered} polynomials of degree $d>1$, i.e.
maps $f:\C\to\C$ of the form $f(z)=z^d+a_{d-2}z^{d-2}+\dots+a_0$
(for instance, $\poly_2$ consists of polynomials $Q_c(z)=z^2+c$).

Evidently, polynomials from $\poly_d$ come in 
groups of conjugate maps
(where the conjugation is the multiplication by any $(d-1)$-st root of unity).
This creates the symmetry, with respect to such multiplication, of the space.
The \emph{connectedness lo\-cus} $\M_d$ of $\poly_d$ is the set of all $f\in\poly_d$ with connected $K_f$.
The \emph{principal hyperbolic domain} $\phd_d$ of $\poly_d$ 
consists of all $f\in \poly_d$ with an attracting fixed point $a$ and all critical points in the immediate basin of $a$
(then the immediate basin of $a$ is an invariant Jordan disk $U$ while $\ol{U}=K_f$ is a closed Jordan disk).
Topologically, $\phd_d$ is an open ball in the $(d-1)$-dimensional complex space, symmetric under the rotation by
any $(d-1)$-st root of unity.
In particular, $\M_2$ identifies with the celebrated \emph{Mandelbrot set} $\Mc=\{c\in\C\mid Q_c^n(0)\not\to\infty\}$,
and the boundary of $\phd_2$ is called the \emph{Main Cardioid}. This paper is inspired by the work of Thurston \cite{thu85}
who introduced the notion of an invariant lamination and used 
it to construct a model of $\M_2$.

From the topological point of view, 
 transparent dynamics takes place on locally connected
(and therefore connected) Julia sets. Indeed, if $J_P$ is locally connected then,
by  Carathe\-odory theory,  all external rays land. In this case
one can introduce the concept of a \emph{laminational equivalence relation $\sim_P$} under which the arguments of external rays landing on
the same point are declared to be equivalent. The corresponding \emph{topological
polynomial $f_{\sim_P}$ of degree $d$} is the map induced
by $\si_d$, the angle $d$-tupling map, on $\uc/\sim_P$;
the continuum $\uc/\sim_P$ is called the \emph{topological Julia set} of $f_{\sim_P}$ and is denoted $J_{f_{\sim_P}}$, or $J_{\sim_P}$.
Since $J_P$ is locally connected, then $f_{\sim_P}$ and $P|_{J_p}$ are topologically conjugate.

The notion of a \emph{laminational equivalence relation $\sim$} can be introduced without referring to any polynomial. It is an equivalence
relation on the unit circle $\uc$ which has all the essential properties of $\sim_P$ described in the previous paragraph. Moreover, for an
\emph{invariant laminational equivalence relation $\sim$} one can define the associated \emph{topological Julia set $J_\sim=\uc/\sim$}, and the \emph{topological polynomial}
$f_\sim:J_\sim\to J_\sim$ induced on $J_\sim$ by the action of $\si_d:\uc\to \uc$.

For $A\subset \C$ we define the (closed) \emph{convex hull} $\ch(A)$ as the intersection of all \emph{closed} convex sets containing $A$.
Thus, $\ch(A)=\ch(\ol{A})$ by definition.
Given an invariant laminational equivalence relation $\sim$ one can consider the family $\lam_\sim$ of edges of the convex hulls of
all $\sim$-classes. More precisely, by construction, $\lam_\sim$ is the family whose elements are (1) edges of convex hulls of $\sim$-classes,
and (2) all points
of the unit circle $\uc$; the chords forming $\lam_\sim$ are said to be \emph{leaves (of  $\lam_\sim$)}.
Thurston's laminations $\lam_\sim$ associated with laminational equivalence relations $\sim$ (or with their
topological polynomials $f_\sim$) provide a geometric visualisation for the latter and are called
\emph{q-laminations}. In fact, as topological polynomials, q-laminations, and laminational equivalence relations are in one-to-one correspondence, we
talk about them interchangeably.

Studying topological polynomials is motivated by the fact that
they can be viewed as an ideal version of complex polynomials. The interplay
between the two notions can be key in constructing models of connectedness loci. Indeed,
denote the \emph{set} of all topological polynomials of degree $d$ by $\tpoly_d$. Observe that the set of
topological polynomials does not have a natural topology because they are not defined on the same domain. This gives rise to
the problem of finding a natural way of defining a topology on the set of topological polynomials of degree $d$. The hope here
is that when it is done, the resulting \emph{space} of topological polynomials will serve a model for the corresponding
space of polynomials, i.e. for the corresponding connectedness locus.

By a \emph{model} in this setting we mean that there
exists a \emph{monotone continuous} map from the degree $d$ connectedness locus to the space of topological polynomials of degree $d$
(i.e., to the set $\tpoly_d$ endowed with the appropriate topology) with fibers of the modeling map identifying polynomials
with similar dynamics. Finding such a model is a challenging task, so it is all the more surprising and inspiring that it was fulfilled
by Thurston in his seminal preprint \cite{thu85} in the quadratic case. 

Polynomials $Q_c(z)=z^2+c$ are parameterized by $c$, i.e. by their critical values. The role of a critical point in a quadratic
q-lamination $\lam_\sim$ is played by the unique critical set of $\lam_\sim$, at least in the case when the Julia set $J(Q_c)$
is locally connected and the critical point $0$ belongs to $J(Q_c)$. The assumption $0\in J(Q_c)$ rules out certain important
classes of polynomials, namely those with attracting/parabolic periodic points. However, the polynomials with locally connected
Julia sets and no attracting/parabolic points still remain available in the above sense. For them one can define
the associated invariant laminational equivalence relation $\sim_c$ and the associated invariant q-lamination $\lam_{\sim_c}=\lam_c$.

Since we exclude (for now) polynomials with attracting/pa\-ra\-bo\-lic periodic points, the associated topological polynomials
$f_c$ will not have \emph{periodic Fatou domains of degree greater than $1$} (called in what follows
\emph{periodic hyperbolic domains}). Then the topological polynomial $f_c$ will have a unique critical point in $J(f_c)$ associated with
a finite $\sim_c$-class $C_c$ and 
its convex hull $\ch(C_c)$ from $\lam_c$. Mimicking complex polynomials $Q_c$ that are
tagged with their critical values, one can tag $f_c$ with the set $\ch(\si_2(C_c))$. Alternatively, one can use the set
$\ch(C_c)$ itself. Observe that since $\ch(C_c)$ is critical, one way of characterizing $\ch(C_c)$ is by listing all critical leaves
(in the quadratic case these are \emph{diameters}) contained in $\ch(C_c)$.
One can show that 
 diameters of all $\ch(C_c)$ for all $c$ are precisely
 diameters with non-periodic endpoints.
Call them \emph{non-periodic diameters}; such
 diameters are dense in all diameters.
So, to each  set $\ch(C_c)$ as above
we associate the collection of diameters contained in $\ch(C_c)$. These collections \emph{partition} the space of
non-periodic diameters. Moreover, this partition is upper semicontinuous \emph{in itself} in the following sense: if
non-periodic diameters $\oc_i$ and $\oy_i$ 
are contained in one critical set $\ch(C_c)$, $\oc_i\to \oc,$ $\oy_i\to \oy$, and $\oc$ is non-periodic,
then so is $\oy$, and, moreover, $\oc$ and $\oy$ 
are contained in one set $\ch(C_t)$ for some complex parameter $t$. This generates the quotient space of the space of non-periodic diameters.
Observe that to see that the quotient space is planar
one should use the sets $\ch(\si_2(C_c))$.

The idea is to extend this equivalence onto all diameters by considering limits of classes as new classes of equivalence. This
results into a transparent family of classes of equivalence among all diameters, associated with
topological polynomials in the one-to-one fashion and yields the desired
creation of the space of quadratic topological polynomials. An amazing achievement by Thurston \cite{thu85}
was to prove that this space of quadratic topological polynomials can serve as a monotone model of the boundary of the Mandelbrot set $\M_2$.
The famous MLC conjecture states that this model is actually homeomorphic to $\M_2$.

\section{Problem statement and  main results}

In this paper, we move towards implementing the same program and set up a framework for
that in the higher degree case. We approach the problem from a slightly different angle. An analog of the diameters
in the higher degree $d$ cases are the so-called \emph{critical portraits}, i.e. collections of $d-1$
distinct $\si_d$-critical leaves which are
pairwise disjoint inside the open unit disk and whose union contains no loops (if loops
are allowed then there will be room to enter at least one more critical chord disjoint from the already given ones which means
that the given critical chords do not exhaust the criticality of the map while we want critical portraits to completely
define criticality).


We consider \emph{all}
critical portraits $\Po$ and develop tools allowing us to associate with them specific q-laminations which we call \emph{root q-laminations}
and denote them by  $\lam^{root}_\Po$.

In the process we aim to clarify the connection between q-laminations and critical portraits.

\begin{dfn}\label{d:introlegal} Given a q-lamination $\lam$ a critical portrait $\Po$ is
\emph{legal} for $\lam$ (and vice versa) if there exists a proper lamination $\lam'$ generating $\lam$ such that:
\begin{enumerate}
\item no chord in $\Po$ crosses any leaf of $\lam'$, and 
\item  for every  critical chord of $\Po$ with a periodic endpoint  $x=x_0$ of period $n$, there exists a cycle of hyperbolic gaps $U_0,\ldots,U_{n-1}$ of $\lam'$, also of period $n$, so that
if $\si_d|_{\bd(U_i)}$ has degree $d_i$, then $d_i-1$ critical chords of $\Po$ are contained in $U_i$ and have a convex hull which is
an all-critical polygon with a vertex $\si_d^i(x_0)$.
\end{enumerate}
We also say that $\Po$ and $\lam'$ are legal to each other and that
$(\Po, \lam)$ and $(\Po, \lam')$ are \emph{legal pairs}.
\end{dfn}

As was mentioned above, in the degree two case there is only one q-lamination that is legal for a critical portrait $\Po$.  But, if $d\ge 3$, then there may be
multiple q-laminations legal for $\Po$. In the construction of the root lamination associated with $\Po$ we introduce the notion of a
\emph{web}: a collection of chords in the disk that may cross each other but
 otherwise satisfy the properties of an invariant lamination.
This allows us to consider all possible iterated pullbacks of chords in $\Po$,
 which form a web.

In this way we construct the smallest lamination legal for $\Po$, and call it the \emph{root} lamination of $\Po$.
It follows that all other q-laminations legal for $\Po$ can be obtained from  $\lam_\Po^{root}$ by adding leaves in
hyperbolic gaps of $\lam_\Po^{root}$ in specific ways and, as a result, that there are only finitely many
such q-laminations.

We then consider special class of critical portraits, namely \emph{critical portraits whose chords have non-periodic endpoints}.
To explain the upcoming terminology
consider q-laminations that do not have periodic hyperbolic gaps. Their topological Julia sets may be dendrites (i.e., locally connected
planar continua that contain no Jordan curves) in which case the corresponding q-lamination is said to be \emph{dendritic},
and any critical portrait $\Po$ compatible with $\lam$ is said to be \emph{dendritic}.

Otherwise, i.e., if the Julia sets contain Jordan curves, these curves are mapped forward
onto their image Jordan curves a few times until the eventual image Jor\-dan curve is periodic with the first return map being an irrational rotation.
Recall that such curves can be a part of a polynomial Julia set in which case they are boundaries of the so-called \emph{Siegel domains}
inside the filled Julia set containing neutral periodic points of the same period with irrational argument of the multiplier.
This explains why we call  q-laminations like that  \emph{\textbf{Si-}egel-\textbf{de-}ndritic}, or simply \emph{Side} q-laminations.
A \emph{Side} critical portrait is a critical portrait  all of whose  leaves have non-periodic endpoints.
We prove that a critical portrait is Side if and only if there exist a Side q-lamination compatible with it.
For each Side critical portrait $\Po$ we show that its \emph{root lamination} $\lam^{root}_{\Po}$
is the only q-lamination \emph{with finite critical sets} compatible with $\Po$. Varying the critical portrait
$\Po$ we can get all q-laminations with finite critical sets.

\section{Laminations}\label{s:lami-gen}

Set $\uc=\{z\in\C\,|\,|z|=1\}$. We parameterize the external
rays of a polynomial $f\in\M_d$ by 
elements of $\R/\Z$ (so, the full angle is $1$ rather than $2\pi$).
The external ray of argument $\theta\in\R/\Z$ is denoted by $R_f(\theta)$;
$f$ maps $R_f(\theta)$ to $R_f(d\theta)$. Let $\si_d:\uc\to \uc$ be the self-mapping of $\uc\subset \C$ that takes $z$ to $z^d$
(if we represent $\uc$ as $\R/\Z$ then $\si_d(t)=d\cdot t\, \mod\, 1$).
A \emph{chord} $\ol{ab}$ is a closed segment connecting points $a$, $b$
of $\uc$, often represented by their arguments
(thus, $\ol{\frac 13\frac 23}$ is the chord connecting the points $e^{2\pi i/3}$ and $e^{4\pi i/3}$).
Two chords \emph{cross} if their endpoints alternate on $\uc$; an important class of
families of chords are \emph{prelaminations}, i.e. families of chords that pairwise do not cross; if such a family is closed then it is called
a \emph{lamination}. In this context chords from a (pre)lamination are called \emph{leaves}.

Convex hulls of closed nowhere dense subsets $T$ of $\uc$ are called \emph{polygons} and are denoted by $\ch(T)$ (without assuming
that the cardinality of $T$ is finite). In the context of laminations, if such a set $T$ consists of more than two points, then it convex
hull is called a \emph{gap} provided all boundary chords of $\ch(T)$ are leaves of a given lamination $\lam$ and no leaf of $\lam$ enters the interior
of $\ch(T)$. Moreover, polygons are called \emph{finite, infinite, countable, uncountable} depending on the cardinality of $T$.
We define the $\si_d$-image of a chord $\ol{ab}$ as the chord $\ol{\si_d(a) \si_d(b)}$, i.e. through the action of $\si_d$
on $a$ and $b$ (when we say that we apply $\si_d$ to a chord $\ol{ab}$ we apply it to the endpoints of $\ol{ab}$).

\subsection{Laminational equivalence relations}\label{ss:lclam}

For $f\in\M_d$ let $\psi$ be the Riemann map $\psi:\mathbb C\setminus \overline{\mathbb D}\to\mathbb C\setminus K_f$ with $\psi'(\infty)>0$
so that $f\circ \psi(z)= \psi(z^d)$.
For $f\in\M_d$ with locally connected $J_f$, the map $\psi$ can be continuously extended over the boundary of $\mathbb C\setminus \overline{\mathbb D}$
so that for the extended map $\ol{\psi}$ we can define $\ol{\psi}(e^{2\pi i\theta})$,
the landing point of $R_f(\theta)$;  then $\ol{\psi}:\uc\to J_f$ is a semi-conjugacy
between $\si_d:\uc\to\uc$ and $f:J_f\to J_f$
called the \emph{Caratheodory loop}.
Define an equivalence relation $\sim_f$ on $\uc$ as follows: $x \sim_f y$ if and only if $\ol{\psi}(x)=\ol{\psi}(y)$
and call $\sim_f$ the \emph{laminational equivalence relation (generated by $f$)}.
The relation $\sim_f$ is $\si_d$-invariant; $\sim_f$-classes have pairwise disjoint convex hulls.
The quotient space $\uc/\sim_f=J_{\sim_f}$ is called a \emph{topological Julia set}.
Clearly, $J_{\sim_f}$ is homeomorphic to $J_f$.
The map $f_{\sim_f}: J_{\sim_f}\rightarrow J_{\sim_f}$, induced by $\sigma_d$ and
called a \emph{topological polynomial}, is topologically conjugate to $f|_{J_f}$.

Equivalence relations analogous to $\sim_f$ can be introduced
with no reference to polynomials (see \cite{hubbdoua85, thu85}, see also
\cite{bl02}). Let $\sim$ be an
equivalence relation on $\uc$. Equivalence classes of $\sim$ will
be called \emph{($\sim$-)classes}.

\begin{dfn}\label{d:lameq}
An equivalence relation $\sim$ on $\uc$ is a \emph{($\si_d$-) invariant
laminational equivalence relation} if it is:

\begin{enumerate}
\item[(E1)]
\emph{closed}: the graph of $\sim$ is a closed
set in $\ucirc \times \ucirc$;

\item[(E2)] 
\emph{unlinked}: if $\g_1$ and $\g_2$ are distinct $\sim$-classes,
then their convex hulls $\ch(\g_1), \ch(\g_2)$ are disjoint in $\ol{\bbd}$;

\item[(E3)] \emph{finite}: all $\sim$-classes are finite;

\item[(D1)] 
\emph{forward invariant}: for a class $\g$,
the set $\si_d(\g)$ is a class too;


\item[(D2)] 
\emph{backward invariant}: for a class $\g$,
its preimage $\si_d^{-1}(\g)=\{x\in \ucirc: \si_d(x)\in \g\}$ is a
union of classes;

\item[(D3)] \emph{orientation preserving}: for any $\sim$-class $\g$ with more than two points, the
map $\si_d|_{\g}: \g\to \si_d(\g)$ is a \emph{covering map with
positive orientation}, i.e., for every connected component $(s, t)$ of
$\ucirc\setminus \g$ either $\si_d(s)=\si_d(t)$ or  the arc in the circle $(\si_d(s), \si_d(t))$ is a
connected component of $\ucirc\setminus \si_d(\g)$.
\end{enumerate}

Here, conditions (E1) -- (E3) are requirements on the  \textbf{E}qui\-va\-lence relation, while
conditions (D1) -- (D3) deal also with the \textbf{D}ynamics of $\si_d$.
Note that (D1) implies (D2).
\end{dfn}

Let us introduce two extreme equivalence relations.

\begin{dfn}[Full and empty]\label{d:fullempty}
The equivalence relation such that all points of the circle are equivalent is said to be \emph{full}.
The equivalence relation such that no distinct points of the circle are equivalent is said to be \emph{empty}.
\end{dfn}

(Invariant) laminational equivalence relations have visual counterparts (in what follows the degree $d\ge 2$ is fixed).

\begin{dfn}\label{d:q}
Let $\sim$ be an (invariant) laminational equivalence relation. Then the family of all edges of the convex hulls of $\sim$-classes,
together with all points of $\uc$, is called an \emph{(invariant) q-lamination (generated by $\sim$)} and is
denoted by $\lam_\sim$; chords in $\lam_\sim$ are said to be the \emph{leaves} of $\lam_\sim$.
\end{dfn}

It follows that $\lam_\sim$ and $\sim$ completely define each other. However,
the construction of $\lam_\sim$ makes sense for a much wider class of equivalence relations, namely for those satisfying
requirements (E1) and (E2) of Definition \ref{d:lameq}. If we do not require (E3), i.e. allow for infinite classes, then
$\lam_\sim$ no longer defines $\sim$. For example
the full equivalence relation $\sim_1$ and the empty equivalence relation $\sim_2$ generate 
the same sets $\lam_{\sim_1}$ and
$\lam_{\sim_2}$, and the pictures look the same, i.e. as the circle
with no chords in it. Sometimes extremes meet.

\subsection{General properties of laminations}

Thurston \cite{thu85} defined \emph{invariant laminations} as families of chords
with dynamical properties like those of q-laminations but weaker.
One reason for that was the fact that some polynomials have connected
but not locally connected Julia sets. However the idea is to associate combinatorial dynamic structures to \emph{all polynomials},
including those with non-locally connected Julia sets.
A natural way of handling this is to approximate a polynomial $P$
by polynomials with locally connected Julia sets, and then to associate to $P$ the Hausdorff limit point(s)
of the corresponding sequence of q-laminations. Such limits have properties similar to those of q-laminations
but  may not be q-laminations.  Thurston's definition aims at taking care of this issue using geometric
and dynamic properties of chords forming laminations.

\begin{dfn}[Laminations]\label{d:geolam}
Two distinct chords \emph{cross} (or are \emph{lin\-ked}) if they intersect
inside $\disk$ (equivalently, if their endpoints alternate on $\uc$).
A \emph{prelamination} is a family of chords $\lam$ cal\-led
\emph{leaves} such that distinct leaves are unlinked and all points of
$\uc$ are (degenerate) leaves. If the set
$\bigcup\lam$ is compact, then $\lam$ is called a
\emph{lamination}. Two (pre)laminations are \emph{compatible} if their leaves
do not cross (thus, the union of two compatible (pre)laminations is a (pre)lamination).
\end{dfn}

Observe that, formally speaking, a set, consisting of one chord and all individual points of $\uc$,
is already a prelamination. 
From now on, $\lam$ always denotes a (pre)lamination.

\begin{dfn}[Gaps and edges]\label{d:gaps} \emph{Gaps} of a lamination $\lam$ are the closures of
components of $\disk\sm \bigcup\lam$. A gap $G$ is \emph{countable $($finite,
un\-coun\-table$)$} if $G\cap\uc$ is countable infinite (finite, un\-coun\-table).
Uncountable gaps are called \emph{Fatou} gaps.
For a closed nowhere dense set $X\subset \uc$, consider its convex hull $\ch(X)$; then
\emph{edges} of $\ch(X)$ are maximal straight non-degenerate segments in $\bd(\ch(X))$,
and vertices of $\ch(X)$ are points of $X$. Some additional terminology (described in what follows) will
apply to invariant laminations.
\end{dfn}

Convergence of (pre)laminations $\lam_i$ to a set of chords $\mathcal E$
is understood as the Hausdorff metric convergence of leaves of $\lam_i$ to chords from $\mathcal E$
so that all limits, in terms of the Hausdorff metric, of sequences of chords from $\lam_i$ form $\mathcal E$.
Evidently, $\mathcal E$ is a lamination. A lamination $\lam$ is \emph{nonempty} if it
has nondegenerate leaves and \emph{empty} otherwise
(the empty lamination is denoted by $\lam_\0$; note that it is not the empty set
as it contains all points of $\uc$). Say that $\lam$ is \emph{countable}
if it has countably many nondegenerate leaves and \emph{uncountable}
otherwise; $\lam$ is \emph{perfect} if it has no isolated leaves. 
In this paper by \emph{countable} we always mean \emph{infinite countable}.

If $G\subset\cdisk$ is the convex hull of $G\cap\uc$, let $\si_d(G)$
be the convex hull of $\si_d(G\cap\uc)$. Observe that this is not the action of a map from
$\cdisk$ to $\cdisk$ restricted on the set $G$. Rather, it is a map from the family of
convex hulls of closed subsets of $\uc$ to itself. The most important particular case here is
when 
$G=\ol{ab}$ is a chord. In that case, $\si_d(G)=\si_d(\ol{ab})=\ol{\si_d(a)\si_d(b)}$,
and we call $\ol{ab}$ a \emph{pullback chord} of $\ol{\si_d(a)\si_d(b)}$.
In some cases it is useful to consider a special extension of $\si_d$ onto $\cdisk$.
The map $\si_d$ can be extended continuously over $\bigcup\lam$ 
so that it acts linearly on every leaf of $\lam$.
We also denote this extended map by $\si_d$.

A map $m:X\to Y$ of a topological space $X$ to a topological space $Y$ is \emph{monotone} if for each $y\in Y$, $m^{-1}(y)$ is connected.

\begin{dfn}\cite[Thurston Invariant Lamination]{thu85}\label{dfn-Thurston}
A la\-mination $\mathcal{L}$ is {\em Thurston ($\si_d$)-invariant} if it
satisfies the following conditions.

\begin{enumerate}

\item Forward $d$-invariance: for any leaf $\ell=\overline{pq} \in
    \mathcal{L}$, 
    we have that $\overline{\sigma_d(p)\sigma_d(q)}=\si_d(\ell) \in \mathcal{L}$.

\item Backward invariance: for any leaf $\overline{pq} \in
    \mathcal{L}$, there exists a collection of $d$ {\bf disjoint}
    leaves in $\mathcal{L}$ (this collection of leaves may not be
    unique), each joining a pre-image of $p$ to a pre-image of $q$.

\item Gap invariance: For any gap $G$, the convex hull $H$ of $\si_d(G\cap\uc)$
is a gap or a leaf. 
If $H$ is a gap, $\si_d|_{\bd(G)}:\bd(G)\to\bd(H)$ maps as the
composition of a monotone map and a covering map to the boundary of
the image gap, with positive orientation (the image of a point
moving clockwise around $\bd(G)$ must move clockwise around the
image $\bd(H)$ of $G$).
\end{enumerate}

\end{dfn}

By \cite[Lemma 2.1]{bl02}, \emph{any q-lamination is Thurston invariant}.

\section{(Sibling) invariant laminations and webs}\label{ss:sib-prop}

In Definition~\ref{dfn-Thurston} one deals with gaps, i.e. closures
of components of the complement $\disk\sm \bigcup \lam$. This is a direct
way to ensure that $\si_d$ has a nice extension over the plane. A
drawback of Definition \ref{dfn-Thurston} is that while $\lam$ otherwise is defined
as a family of chords of $\disk$ (leaves), in gap invariance we talk
about different objects (gaps). As a result, it is hard to verify gap invariance
for specific collections of chords even though other properties listed in
Definition~\ref{dfn-Thurston} are easy to check.

A different approach was suggested in \cite{bmov13}, largely borrowing terminology (and inspiration) from \cite{thu85}.
This new definition has the advantage of dealing only with chords (leaves); this makes verifications easier.
We call collections of leaves satisfying the new definition  \emph{(sibling) invariant laminations} and prove
in \cite{bmov13} that they all are Thurston invariant.

To simplify navigating through the notation we adopt the agreement
according to which families of chords (more generally, families of sets) are denoted by calligraphic letters
(for example, $\Ac,$ $\Bc,$ $\Nc$, $\lam$ etc). The union of
the corresponding chords (sets) will then be denoted by $\bigcup \Ac,$ $\bigcup \Bc,$ $\bigcup \Nc$, $\bigcup \lam$ etc.
Moreover, the terminology for families of chords (sets) and their unions will be related, too. For example, if $\lam$ is a lamination, then $\bigcup \lam$
will be called a \emph{lamination(-al) set}. Recall that a leaf $\ell^*\in \lam$ with $\si_d(\ell^*)=\ell$ is called a \emph{pullback} of $\ell$ (within $\lam$).

\begin{dfn}[Sibling property and webs]\label{d:sipro}
Fix an integer $d>1$.  A \emph{sibling of a chord $\ell$} is a
chord $\ell'$ different from $\ell$ with $\si_d(\ell')=\si_d(\ell)$. Also, two points with the same $\si_d$-image are
called \emph{sibling points} or just \emph{siblings}.
A family $\Wc$ of chords is said to have the \emph{sibling property (for $\si_d$)}
if for any \textbf{non-critical} chord $\ell\in \Wc$ there exist $d$ \textbf{pairwise disjoint} chords $\ell_1$, $\dots$, $\ell_d$ in $\Wc$ such that
$\ell_1=\ell$ and $\si_d(\ell_1)=\dots=\si_d(\ell_d)$. Any collection
of chords $\{\ell_1$, $\dots$, $\ell_d\}$ with these properties is called a \emph{(full) sibling collection}.
Any family of chords with sibling property is called a \emph{$(\si_d$-)web}. 
\end{dfn}

Chords in a sibling collection cannot intersect even on $\uc$. On the other hand, if $\Wc$ is a web,
then Definition \ref{d:sipro} allows for the possibility that chords in $\Wc$ cross each other. In fact, any web can be
obtained as follows: take a family of sibling collections and then the set of all chords of all these collections.
According to the general principle, if $\bigcup\Wc$ has a topological property
(for example, connected, closed etc) then the web is called accordingly (for instance, \emph{connected web}, \emph{closed web} etc).
The union of chords from a web is called a \emph{web set}. 

A web does not have to have any invariance properties under the action of $\si_d$. Rather, the web establishes a certain level of symmetry
that a family of chords with sibling property must possess. In the quadratic case (i.e., when $d=2$) a chord is critical if and only if
it is a diameter of $\cdisk$. Then the sibling chord of a non-critical chord $\ell$ is a $\frac12$-rotation of $\ell$, and a set of chords is a $\si_2$-web
if and only if it is the set of chords invariant under the rotation by $\frac12$.

Clearly, the assumption that $\ell$ is non-critical in Definition \ref{d:sipro} is a must. Indeed, if $\ell$ is $\si_d$-critical, then its sibling chords
are also critical and map to the same point, say, $x\in \uc$. The point $x$ has, overall, $d$ preimages. Hence the maximal number of disjoint (critical) leaves that map to
$x$ is the integer part $\lfloor \frac{d}2 \rfloor$ of $\frac{d}2$ which is less than $d$. Thus, we cannot have $d$ disjoint chords that have
the same degenerate image and a sibling collection of a $\si_d$-critical leaf does not make sense.

We endow the family of all sibling collections with the natural topology and talk about the
\emph{space} of sibling collections. This space is not compact as some sequences of sibling collections
can converge to collections of critical chords (just take sibling collections whose image chords converge to a point).
Still, the following lemma holds.

\begin{lem}\label{l:clos-sibl-1}
The closure of a web is a web.
\end{lem}

\begin{proof}
We need to show that if a sequence of sibling collections converges while their image chords do not converge to a point
then the limit collection of chords is a sibling collection.
Let $\Zc=\{\ell_1, \dots, \ell_d\}$ be the limit of a sequence of sibling collections. Then all
chords $\ell_i, 1\le i\le d$ have the same non-degenerate image.
By continuity, $\ell_1, \dots, \ell_d$ are pairwise disjoint except, perhaps, for the endpoints.
We claim that $\ell_1$ and $\ell_2$ are disjoint. Indeed, otherwise we may assume that $\ell_1=\ol{ab}$ and $\ell_2=\ol{bc}$
and $\si_d(a)=\si_d(c)\ne \si_d(b)$.
A sibling collection near $\Zc$ must include disjoint leaves $\ol{a'b'}\approx \ol{ab}$ and $\ol{b''c'}\approx \ol{bc}$ with
the same image which is impossible because $\si_d(b')\ne \si_d(b'')$, a contradiction. Similarly we can show
that all chords $\ell_1, \dots, \ell_d$ are pairwise disjoint as desired.
\end{proof}

The sibling condition is not very restrictive.

\begin{lem}\label{l:all-aboard}
Any non-critical chord $\ell=\ol{xy}$ can be included in a sibling collection.
\end{lem}

\begin{proof} Left to the reader.
\end{proof}

\subsection{Invariant webs and laminations}\label{sss:invaweb}
Let us add dynamics to the just introduced concepts.

\begin{dfn}[Invariant webs]\label{d:web}
A web $\wb$ is \emph{($\si_d$)-invariant} if it has the following additional properties:

\begin{enumerate}

\item if $\oc\in\wb$, then $\si_d(\oc)\in\wb$ 
    and all points of $\uc$ are included in $\wb$;

\item for each $\oc\in \wb$ there exists $\oq\in \wb$
    so that $\si_d(\oq)=\oc$.


\end{enumerate}
\end{dfn}


Invariant webs survive the operation of taking the closure.

\begin{lem}\label{l:closke}
It $\wb$ is an invariant web, then its closure $\ol{\wb}$ is a closed invariant web.
The union of any collection of invariant webs is an invariant web.
\end{lem}

\begin{proof}
To prove the lemma we check properties (1) and (2) from Definition \ref{d:web} using
the continuity of $\si_d$.

(1) If $\oc\in\ol{\wb}$, then, clearly, $\si_d(\oc)\in\ol{\wb}$ or $\si_d(\oc)$ is a point.

(2) For any $\oc\in \ol{\wb}$ choose chords $\ell_i\in \wb$ such that $\ell_i\to \oc$, then their pullbacks
within $\wb$, and then a subsequence of the pullbacks converging to a chord from $\ol{\wb}$ which maps to $\oc$.

Together with Lemma \ref{l:clos-sibl-1} this proves the first claim. The second claim is immediate.
\end{proof}

Definition \ref{d:sibli}
introduces invariant laminations which are a particular case of an invariant web.
This concept extends that of a lamination similar to how the concept
of an invariant web extends that of a web.

\begin{dfn}\cite[Definition 3.1]{bmov13}\label{d:sibli}
A prelamination $\lam$ is \emph{($\si_d$-)in\-va\-ri\-ant} if it is an invariant web.
A closed ($\si_d$-)in\-va\-ri\-ant prelamination $\lam$ is called a \emph{($\si_d$)-in\-va\-ri\-ant}
lamination.
\end{dfn}

Note that if $\{\lam_\al\}_{\al\in \Af}$ is a family of invariant prelaminations
totally ordered by inclusion then $\ol{\bigcup_{\al\in \Af} \lam_\al}$ is an invariant lamination.

While superficially (invariant) webs are, in a sense, opposite of (invariant) laminations (chords of webs may
well cross each other inside $\disk$), they are a useful tool in studying invariant laminations
as they allow one to simultaneously consider various laminations compatible with a given prelamination.

Let us explain now which topology we consider on the family of all webs and related families.
The space of chords (all chords in $\disk$, including degenerate chords) with the Hausdorff metric is, evidently a compact metric subspace
of the space $2^{\cdisk}$ of all compact subsets of the closed unit disk.
Two collections of chords are close if the corresponding collections of points in $2^{\cdisk}$ are close.

Theorem \ref{t:laclo} combines parts of Corollary 3.20, Theorem 3.21 and Corollary 3.22 of \cite{bmov13}
into one statement. An alternative proof of 
Theorem \ref{t:laclo} may rely on Lemma \ref{l:closke}.

\begin{thm}\label{t:laclo}
The closure of an invariant prelamination is an invariant lamination.
The space of all sibling invariant laminations of degree $d$ is compact.
\end{thm}

Here is another useful property of invariant laminations.

\begin{lem}\label{l:clos-sibl}
The family of all non-isolated leaves of an invariant lamination $\lam$ is an
invariant lamination.
\end{lem}

\begin{proof}
All non-isolated leaves in $\lam$ form a forward invariant closed
family of leaves. We claim that if $\ell\in \lam$ is non-isolated then it has a preimage $\oq\in \lam$ which is
also non-isolated. Indeed, we can choose leaves $\oq_i\to \oq$ with $\si_d(\oq_i)\to \ell$ so that leaves
$\si_d(\oq_i)$ are pairwise distinct. It follows that $\si_d(\oq)=\ell$, $\oq_i$'s are pairwise distinct, and, hence,
$\oq$ is non-isolated in $\lam$. Now, let $\ell$ be non-isolated and non-critical.
Choose pairwise distinct $\ell_i\to \ell$ and their sibling collections.
We may assume that these sibling collections of leaves converge; by Lemma \ref{l:clos-sibl-1} they converge to a
sibling collection that includes $\ell$.
\end{proof}

Theorem \ref{t:gapin} combines some results of \cite{bmov13} into one.

\begin{thm}\cite[Lemma 3.1, Theorem 3.2]{bmov13}
\label{t:gapin}
Any q-la\-mi\-na\-tion is an invariant lamination.
Any invariant lamination is gap invariant, and, hence, Thurston invariant.
\end{thm}

From now on, unless stated otherwise, all laminations are $\si_d$-invariant for some $d\ge 2$.

\begin{lem}\label{l:samep}
If $\lam_\sim$ is a invariant q-lamination, $x$ is a periodic point of period $n$, and $\ol{xy}$ is a leaf
of $\lam_\sim$, then $y$ is also of period $n$. Thus, if a chord $\ell$ eventually maps to a chord that connects an $n$-periodic
point with a point which is not $n$-periodic then $\ell$ cannot be a leaf of any q-lamination.
\end{lem}

\begin{proof}
Since classes of laminational equivalence relations are finite and $x$ is periodic, then $y$ is periodic as otherwise
properties of invariant laminations imply that $\ol{xy}$ can be pulled back $n$ times to give rise to a leaf $\ol{xz}$ with
$\si^n_d(x)=x$ and $\si^n_d(z)=y$; clearly, $\ol{yz}\ne \ol{xy}$ is concatenated with $\ol{xy}$, and repeating the pullback procedure we will
obtain an infinite cone of leaves of $\lam_\sim$, and hence an infinite $\sim$-class, a contradiction.
On the other hand, by \cite[Lemma 2.25]{bopt20}
$y$ is (pre)periodic of period $n$. Hence $y$ is in fact periodic of period $n$ as claimed.
\end{proof}

\subsection{Proper laminations}\label{ss:propela}

As was remarked above, the origin of the concept of a lamination is the construction from Definition \ref{d:q}.
For a given laminational equivalence $\sim$ this construction produces the associated q-lamination $\lam_\sim$
that consists of the edges of the convex hulls of $\sim$-classes. 

On the other hand, invariant laminations appear naturally as Hausdorff limits of q-laminations often reflecting
convergence of the corresponding polynomials. Therefore, it is desirable to construct a laminational equivalence
relation based on a given invariant lamination (i.e., not necessarily an invariant q-la\-mi\-na\-tion).
This requires reversing of the process described in Definition \ref{d:q}. Below, in
Definition \ref{d:lamtoeq}, we describe how, for a given invariant lamination, a natural laminational equivalence relation can be defined.
In doing so we follow \cite{bmov13}.

\begin{dfn}\cite[Definition 2.3]{bmov13}\label{d:lamtoeq}
Let $\lam$ be a lamination. Define the equivalence relation
$\approx_\lam$ by declaring that $x{\approx_\lam}y$ if and only if
there exists a finite concatenation of leaves of $\lam$ joining $x$
to $y$.
\end{dfn}

Clearly, $\approx_\lam$ is an equivalence relation. However it may well happen that $\approx_\lam$ is \emph{not}
a laminational equivalence. Reasons for that is the fact, that in many cases $\approx_\lam$ will not be closed and/or will have infinite
classes. The concept is still valid, but it is the most useful in the case when $\lam$ has additional properties guaranteeing that
$\approx_\lam$ \emph{is} a laminational equivalence.

\begin{dfn}\label{d:proper-chord}
A chord $\ell$ is said to be \emph{proper} if no image of $\ell$ connects an $n$-periodic point and a point
which is not $n$-periodic; otherwise $\ell$ is said to be \emph{improper}. An invariant lamination is \emph{proper} if all
its leaves are proper.
\end{dfn}

A useful fact about proper laminations is established in \cite{bmov13}.

\begin{lem} \cite[Lemma 4.2]{bmov13}\label{l:4.2}
Any q-lamination is proper.
\end{lem}

Proper laminations can be defined differently. If two distinct leaves
of $\lam$ have a common endpoint $v$ and equal images, then their union
is said to be a \emph{critical wedge (with vertex $v$)}.

\begin{lem}\label{l:crit-equiv}
A lamination $\lam$ is proper if and only if the following holds:

\begin{enumerate}

\item no critical leaf of $\lam$ has a periodic endpoint, and

\item no critical wedge of $\lam$ has a periodic vertex.

\end{enumerate}

Moreover, if $\lam$ is proper, then $\approx_\lam$ is a laminational equivalence relation.
\end{lem}

\begin{proof}
By Theorem 4.9 \cite{bmov13} if (1) and (2) hold then $\approx_\lam$ is an invariant laminational equivalence relation.
Hence by Lemma \ref{l:samep} all leaves of $\lam$ are proper.

Let us prove that if $\lam$ is proper then (1) and (2) hold. The fact that $\lam$ cannot contain
a critical leaf with a periodic endpoint easily follows.
Suppose by way of contradiction that $W$ is a critical wedge consisting of leaves $\ol{wa}$ and $\ol{wb}$
with $w$ periodic of period $n$,  $\si_d(a)=\si_d(b)$. 
Then one of the points $a, b$ is not
of period $n$, hence the leaf connecting that point and $w$ is improper, a contradiction with $\lam$ being proper.
The last claim of the lemma follows from Theorem 4.9 \cite{bmov13}.
\end{proof}

Lemma \ref{l:crit-equiv} justifies the following definition.

\begin{dfn}\label{d:gen-prop}
Let $\lam$ be a proper lamination. Then the q-lamination $\lam_{\approx_\lam}$ is said to be \emph{generated} by $\lam$.
\end{dfn}

Lemma \ref{l:prop-clos} is related to Theorem \ref{t:laclo}.

\begin{lem}\label{l:prop-clos}
Let $\lam$ be a (non-empty) proper prelamination. Then the closure $\ol{\lam}$ of $\lam$ is a (non-empty) proper lamination.
\end{lem}

\begin{proof}
By
Theorem \ref{t:laclo}, $\ol{\lam}$ is an invariant lamination. We need to show that it is proper. Suppose that there is a leaf
$\ol{xy}$ of $\ol{\lam}$ where $x$ is $n$-periodic while $\si^n_d(y)\ne y$.
Replacing  $\si_d$ by $\si_d^n$, we may assume that $\si_d(x)=x$ is fixed while $\si_d(y)\ne y$. Set $\ol{xy}=\ell$
and consider several cases.

(1) Let $\si_d(y)\ne x$. Assume that $x<\si_d(y)<y$ in the sense of the positive direction on the circle. Since $\ol{xy}$
is not proper, it cannot belong to $\lam$ and must be the limit of a non-constant sequence of leaves $\ell_i=\ol{x_iy_i}\in \lam$ with
$x_i\to x$ and $y_i\to y$. We may assume that $x_i\approx x$ and $y_i\approx y$ (with $\approx$ meaning ``is very close to'').
Since $\lam$ is proper, $x_i\ne x$ and hence $\si_d(x_i)\ne x_i$. Clearly,
either $x<x_i<\si_d(x_i)$, or $\si_d(x_i)<x_i<x$. Then in the former case $\ell_i$ crosses $\si_d(\ell)$,
and in the latter case $\si_d(\ell_i)$ crosses $\ell$, a contradiction.

(2) Let $\si_d(y)=x$. As before, choose a sequence of leaves $\ell_i=\ol{x_iy_i}\in \lam$ such that $x_i\to x$ and $y_i\to y$.
Consider possible locations of the points $x_i, y_i$ with respect to $x$ and $y$. Observe that
leaves $\ell_i$ do not cross $\ol{xy}$; then it easily follows (for geometric reasons and because
$\si_d$ is locally expanding) that the images of $\ell_i$ will cross leaves $\ell_i$,
a contradiction.
\end{proof}

\section{Critical portraits} \label{s:cripo}

We mentioned in the Introduction that we want to relate  q-laminations and
\emph{critical portraits}. In this section we define critical portraits and study their general properties.
We then consider pullbacks of chords in the setting of \emph{sibling invariant laminations}.

\begin{dfn}[Critical portraits]\label{d:crpt}
A collection $\cpc$ of distinct pairwise unlinked $\si_d$-critical chords is said to be a
\emph{critical collection}. An endpoint of a chord from $\cpc$ is said to be a \emph{vertex} of $\cpc$.
A critical collection is \emph{full} if for every component
$B$ of $\disk\sm \bigcup \cpc$ the map $\si_d|_{\bd(B)\cap \uc}$ is injective (except for the endpoints
of critical edges of $B$). A full critical collection $\cpc$ consisting of $d-1$ 
chords is called a \emph{critical portrait}.
The space of all (unordered) critical portraits of degree $d$ with the natural topology is denoted by $\mathrm{CrP}_d$.
\end{dfn}

The proof of the next lemma is left to the reader.

\begin{lem}\label{l:1d}
In a full critical collection $\cpc$ there is a chord $\oc$ with a circle arc
connecting its endpoints of length $\frac{1}{d}$.
\end{lem}

Let us list basic properties of the just defined concepts.

\begin{lem}\label{l:basic}
For a full critical collection $\cpc$, there are \emph{at least} $d$ components of $\disk\sm \bigcup \cpc$.
A critical collection is full if and only if it contains a critical portrait. 
Given a critical portrait $\Po$, there are \emph{exactly} $d$ components of $\disk\sm \bigcup \Po$.
A critical collection $\cpc$ of $d-1$ chords is a critical portrait if and only if no chords from
$\cpc$ form a loop.
\end{lem}

\begin{proof}
Left to the reader.
\end{proof}

Let us emphasize that, by definition, the order of critical chords in a critical collection $\cpc$ does not matter.

\begin{lem}\label{l:dynama}
Let $A$ be a component of $\disk\sm \bigcup \cpc$ where $\cpc$ is a full critical collection.
Then $A$ is an open Jordan disk whose boundary is the union of some chords from $\cpc$ and
(possibly no) open circle arcs. If at least one of the circle arcs is included,
the map $\si_d$ sends $\bd(A)$ onto the entire unit circle
$\uc$ in a one-to-one order preserving fashion except for the chords of $\cpc$ contained in $\bd(A)$
that collapse to points.
\end{lem}

\begin{proof}
Left to the reader.
\end{proof}

We will use the following definition based on Lemma \ref{l:dynama}.

\begin{dfn}[Components of $\disk\sm \bigcup \cpc$]\label{d:dynama}
Let $A$ be a component of $\disk\sm \bigcup \cpc$ where $\cpc$ is a full critical collection.
If $\bd(A)$ is the union of some critical chords from $\cpc$ we call $A$ \emph{all-critical}.
Otherwise we call $A$ \emph{partially critical}.
\end{dfn}

\section{Pullbacks}\label{s:pullbacks}

Let $\cho$ be the space of all chords in $\cdisk$ with the usual metric
(this includes degenerate chords). The map $\si_d$ is defined on $\cho$ as follows: $\si_d(\ol{ab})=\ol{\si_d(a)\si_d(b)}$;
thus, it acts on the space of chords and sends a collection of chords $\Ac$
to another collection of chords $\si_d(\Ac)$. However, the map $\si_d$ does not act on $\cdisk$.
Recall also, that the sibling property is only stated for non-critical chords
(see Definition \ref{d:sipro}). Finally, if pullbacks are defined for a specific full critical collection $\cpc$
we talk of \emph{$(\si_d, \cpc)$-pullbacks}.
If it does not cause ambiguity (for example, if $d$ and $\cpc$ are fixed), the ``$(\si_d, \cpc)$''-part of the notation may be (partially) omitted
or replaced by the word ``immediate''. Moreover, if instead of $\si_d$ one considers $\si_d^n$ for all non-negative values of $n$ then we will use the word \emph{iterated}
or \emph{eventual} pullbacks.

\begin{center} \textbf{The Description of Pullback Chords} \end{center}

Consider a partially critical component $A$ of $\disk\sm \bigcup \cpc$.
By Lemma \ref{l:dynama} there are several possibilities for various pullback chords
of a chord $\ell=\ol{xy}$ into $\ol{A}$. Let us list such possibilities below.

(i) The endpoints of $\ell$ do not belong to the union of images of critical chords from $\cpc$ that are contained in $\bd(A)$. Then there is a unique
pullback chord of $\ell$ into $\ol{A}$; also, $\ell$ will never intersect pullback chords of $\ell$ into other components of $\disk\sm \bigcup \cpc$.
In particular, if no image of a chord from $\cpc$ is an endpoint of $\ell$, then $\ell$ has $d$ pairwise disjoint $\cpc$-pullback chords
(disjoint from chords of $\cpc$).

(ii) One endpoint of $\ell$ (say, $x$) equals $\si_d(\oc)$ for some critical chord $\oc\in \cpc$ lying on the boundary of $A$; on the other hand, $y$ is not equal to
$\si_d(\od)$ for any other critical chord $\od\in \cpc$ that is contained in $\bd(A)$. Let $X$ be the component of $(\bigcup \cpc)\cap \bd(A)$
that contains $\oc$ and set $X'=X\cap \uc$; let $y'\in \bd(A)$ be such that $\si_d(y')=y$. Evidently, $y'$ is unique, $y'\in \uc\sm \bigcup \cpc$,
and all possible $\cpc$-pullback chords of $\ell$ into $\ol{A}$ are chords that connect $y'$ and points of $X'$.

(iii) We have that $x=\si_d(\oc)$ for some $\oc\in \cpc, \oc\subset \bd(A)$ and $y=\si_d(\od)$ for some $\od\in \cpc, \od\subset \bd(A)$.
Denote by $X$ the component of $(\bigcup \cpc)\cap \bd(A)$
that contains $\oc$ and set $X'=X\cap \uc$; denote by $Y$ the component of $(\bigcup \cpc)\cap \bd(A)$
that contains $\od$ and set $Y'=Y\cap \uc$. All $\cpc$-pullback chords of $\ell$ in $\ol{A}$
are all the chords connecting a point of $X'$ to a point of $Y'$.

\begin{lem}[Sibling property for $\cpc$-pullback chords]\label{l:sipul}
Let $\cpc$ be a 
full critical collection. Let $\ell=\ol{xy}$ be a chord in $\cdisk$.
Then the family $\Bc$ of all $\cpc$-pullback chords of $\ell$ has the sibling property.
\end{lem}

Observe that in the setting of Lemma \ref{l:sipul} $\ell$ can be any chord, in particular $\ell$ does not have to be
compatible with $\cpc$. In other words, one can think of $\ell$ as being a chord of another copy of the unit disk $\disk$, distinct from
the one in which $\cpc$ is contained.

\begin{proof}
The case when neither $x$ nor $y$ belongs to $\si_d(\cpc)$ is already considered in the Description of Pullback Chords (i). So we may assume that
at least one endpoint of $\ell$ (say, $x$) belongs to $\si_d(\cpc)$. To establish the sibling property for the family $\Bc$
we use induction on $d$.

First, assume that $d=2$. Then $\cpc=\{\oc\}=\{\ol{tz}\}$ is a diameter, and by the assumption $\si_d(\oc)=x$. In this case
if a $\cpc$-pullback chord $\od$ of $\ell$ shares an endpoint $z$ with $\oc$ and is contained in $\ol{A}$ where $A$ is a component of $\disk\sm \oc$, then
we can choose the pullback $\oy$ of $\ell$ contained in the closure of the other component $B\ne A$ of $\disk\sm \oc$ sharing an endpoint $t$ with $\oc$.
Clearly, $\od\cap \oy=\0$; since $\od$ was an arbitrarily chosen $\cpc$-pullback chord of $\ell$, this proves the lemma for $d=2$.

We now prove the lemma for some $d\ge 3$ while assuming that the lemma holds for $d-1$.
By Lemma \ref{l:1d} we can choose a critical
chord $\oc=\ol{tz}\in \cpc$ such that the arc $[t, z]$ (with positive direction
within $[t, z]$ being from $t$ to $z$), is of length $\frac{1}{d}$ (and the length of the arc $[z, t]$ is $\frac{d-1}d$).

Now, pinch $\cdisk$ along the chord $\oc$, i.e. identify $t$ and $z$;
this produces a figure eight topological space $T=\pi(\uc)$ where $\pi$ is the associated
quotient map (``pinching map''). The space $T$ is the union of a Jordan curve $E=\pi([z, t])$
and a Jordan curve $F=\pi([t,z])$. Moreover, $E\cap F=\pi(t)=\pi(z)$ is the branchpoint $v$
of $T$. The map $\si_d$ induces two locally one-to-one
maps, $g_E:E\to \uc$ of degree $d-1$ and $g_F:F\to \uc$ of degree one (thus, $g_F$ is a homeomorphism
preserving orientation). Moreover, $\pi$-images of critical chords from
$\cpc$ with endpoints in $[z, t]$ form a full $g_E$-critical collection $\pi(\cpc)$.

Suppose that $\oq$ is a $\cpc$-pullback chord of $\ell$
and show that $\oq$ can be completed to a sibling collection for $\si_d$ of chords compatible with $\cpc$; evidently,
all the chords from the collection belong to $\Bc$.

First we choose, by induction applied to $g_E$, $d-1$ pairwise disjoint topological  $(\pi(\cpc), g_E)$-pullback chords of $\ell$. If $\oq$
has endpoints in $[t,z]$ then the choice of these chords can be arbitrary, but if $\oq$ has endpoints in $[z, t]$ then one of these chords must be
$\pi(\oq)$. In either case the desired choice of topological chords is possible by induction. Denote the collection of these $d-1$ chords by
$\Xc$. Next we develop a strategy of lifting these chords to the original circle using the map $\pi$.

If $\si_d(t)=\si_d(z)$ is not an endpoint of $\ell$, this can be done
uniquely so that the lifted chords are still pairwise disjoint; adding to the resulting $d-1$ pullbacks of $\ell$
the $\cpc$-pullback chord of $\ell$ with the endpoints in $(t, z)$ we get the desired sibling collection of $\cpc$-pullback chords of $\ell$ containing $\oq$.

Suppose that $\si_d(t)=\si_d(z)$ is an endpoint of $\ell$. Then among the $d-1$ chords of $\Xc$ there is one, denoted by $\oy$, whose endpoint is $\pi(t)$. Now consider
cases depending on $\oq$. Suppose that $\oq$ has endpoints in $F$. Then one of them is either $t$ or $z$. Suppose it is $t$. Then 
we choose
the lifting of $\oy$ 
with an endpoint $z$. If however $\oq$ has an endpoint $t$, then we lift $\oy$ to a chord with an endpoint
$z$. The other liftings are well-defined and all of them are pairwise disjoint. This gives rise to a desired sibling collection of chords.

Suppose now that $\oq$ has the endpoints in $E$. Then we choose the liftings of chords of $\Xc$ so that $\oq$ is among the liftings. If $\oq$
has $t$ or $z$ as an endpoint it defines the lifting of $\oy$. Otherwise the choice of the lifting of $\oy$ is either with $t$ as an endpoint,
or with $z$ as an endpoint. In any case once the liftings are defined, we can choose a sibling chord of them with the endpoints in $F$
so that in the resulting collection of chords all chords are pairwise disjoint. This shows that for any choice of $\oq$ we can find a full sibling collection
containing $\oq$ as desired.
\end{proof}

\section{Preroot laminations, or how to bypass portals}\label{s:preroot}

The notion of an invariant web (see Definitions \ref{d:sipro} and \ref{d:web}) is new and mimics that of a
\emph{sibling invariant} lamination (\cite{bmov13}).
However, this notion is too broad for our purposes. For example, the family of all chords is a web. To make webs
more like laminations, we add new restrictions and define webs with regard to full critical collections $\cpc$.
This provides for a natural way to construct an invariant web compatible with $\cpc$.

\subsection{Basic notions}

Recall that by \emph{compatible} chords we mean chords that do not cross each other.

\begin{dfn}[Notation]\label{d:poext2}
Let $\cpc$ be a full critical collection. Then
$\{\Kc_i(\cpc)=\Kc_i\}$ are collections of critical chords forming components of $\bigcup \cpc$
called \emph{critical chordial continua (of $\cpc$)}.
For each $i$ denote by $\Cc_i(\cpc)=\Cc_i$ the family of edges of the convex hull $\ch(\bigcup \Kc_i)$ of the set $\bigcup \Kc_i$.
All endpoints of chords from $\Cc_i$ (or chords from $\Kc_i$) map to the same point denoted $\si_d(\Cc_i)$.
Denote by $\cpc^+$ the collection of all-critical polygons of the form $\ch(\Kc_i)$;
 note that edges of sets from $\cpc^+$ form a full critical collection. Call endpoints of chords from
$\Cc_i$ \emph{vertices} of $\Cc_i$. Call
components of the set $\disk\sm \bigcup(\ch(\Kc_i))$ \emph{($\cpc^+$-)complementary components}
(if $\cpc^+$ is fixed, we omit it from terminology).
\end{dfn}

Observe that any set $\Cc_i$ has at most one periodic vertex. Evidently, sets $\ch(\Kc_i)$ are pairwise disjoint,
and sets $\bigcup \Cc_i$ are pairwise disjoint.
If a chord $\ell\in \cpc$ is disjoint from all other chords of $\cpc$, then the corresponding sets
$\Kc_i$ and $\Cc_i$ coincide with $\{\ell\}$.
Also, $\bigcup \Cc_i=\bd(\ch(\Kc_i))$ for each $i$.
Finally, the boundary of any $\cpc^+$-complementary component is the union of alternating critical chords
from $\cpc^+$ and circle arcs.

E.g., suppose that $\cpc=\Po$ is a critical portrait.
It is easy to see that then $\Po=\Po^+$ if and only if $\Po$ consists of $d-1$ pairwise disjoint critical chords.
Thus, if $\Po$ is a \emph{conic} cubic critical portrait consisting of two critical chords $\ell'=\ol{xy}$ and $\ell''=\ol{xz}$
that share a vertex $x$ then $\Po^+=\{\ell', \ell'', \ol{yz}\}$. 

From now on we work with $\cpc^+$ rather than $\cpc$. Our motivation is as follows.
A critical portrait $\Po$ describes all q-laminations in which endpoints of a leaf from $\Po$ are equivalent
whenever possible, e.g. when the corresponding
critical leaf is proper. Thus, if $\Kc_i$ contains no periodic points then $\ch(\bigcup\Kc_i)$ is contained in the
convex hulls of the associated equivalence classes. Hence we can work with convex hulls of such sets $\Kc_i$.

\begin{dfn}[Portals and proper critical sets]\label{d:im-prop}
A set $\Cc_i$ is a \emph{portal} if it contains a critical leaf with a periodic
endpoint; the set  $\ch(\Cc_i)$ is then called a \emph{portal (continuum)}. Otherwise,
both $\Cc_i$ and $\ch(\Cc_i)$ are said to be \emph{proper}. More generally,
a \emph{portal} is an all-critical set with a periodic vertex.
\end{dfn}

On the first step of our two step construction we use the pullbacks
of the proper critical sets.

\begin{dfn}\label{d:pubaweb}
Let $\cpc$ be a full critical collection.
Denote the collection of all \emph{iterated} $\cpc^+$-pullbacks of chords from \textbf{proper} sets of $\cpc^+$
by $\Yc_{\cpc}$. Say that chords $\ell\in \Yc_{\cpc}$ and $\ell'\in \Yc_{\cpc}$ are
\emph{$\Yc_{\cpc}$-equivalent} if  there is a finite chain
of chords of $\Yc_{\cpc}$ that are consecutively non-disjoint and connect $\ell$ and $\ell'$.
Classes of $\Yc_{\cpc}$-equivalence are called \emph{$\Yc_{\cpc}$-piles}.
\end{dfn}

Observe that even though we use the subscript $\cpc$ in the notation above, and will continue doing so in what follows, we
always use the family $\cpc^+$ of pairwise disjoint all-critical sets, and proceed with the family $\cpc^+$ as the basis of
our constructions.

\begin{lem}\label{l:yweb}
The family $\Yc_{\cpc}$ is an invariant web.
\end{lem}

\begin{proof}
By Lemma \ref{l:sipul}, $\Yc_{\cpc}$ is a web. It is easy to see that it also satisfies the requirements in
Definition \ref{d:web} and, hence,  it is invariant.
\end{proof}

\subsection{Properties of $\Yc_\cpc$-piles}

The non-periodicity implies nice properties of $\Yc_{\cpc}$. Recall that by chords
we normally mean \emph{non-degenerate} chords.

\begin{lem}\label{l:finpil}
There is an integer $M>0$ such that the convex hull of any $\Yc_{\cpc}$-pile set
is a polygon with at most $M$ vertices. 
The number of chords in any $\Yc_\cpc$-pile is at most $\frac{M(M-1)}2$.
\end{lem}

\begin{proof}
One can think of pulling back proper chordial continua from $\cpc^+$ as a whole rather than doing it chord by chord.
By definition, such a pullback into one complementary component is convex. Moreover, if a proper
critical set $\Cc_i$ maps to a vertex of a convex chordial continuum $X$, then the entire pullback of $X$ includes
$\Cc_i$ and slightly modified copies of $X$ attached to $\Cc_i$ along each edge of $\Cc_i$. This can happen at several
proper critical sets $\Cc_i$; it follows that a pullback of $X$ can be thought of as a convex set with some chords inside.
If $X$ is a convex polygon, the number of vertices of its pullbacks can thus increase. However, since there are only finitely many proper
sets $\Cc_i$ that we pull back, the above described growth of the number of vertices of pullback chordial continua
takes place only finitely many times for each of the continua and stops after a large number of pullbacks have been taken.
The last claim is left to the reader.
\end{proof}

Images, preimages and complete pullbacks of piles and pile sets are defined in a natural fashion.

\begin{lem}\label{l:basic1}
The following claims are true.

\begin{enumerate}

\item 
All $\Yc_\cpc$-piles consist of at most $\frac{M(M-1)}2$ chords all of which are proper.

\item $\Yc_\cpc$-pile sets are connected

\item $\Yc_\cpc$-piles include no chords from portals.

\item Edges of the convex hulls of $\Yc_\cpc$-pile sets belong to the  corresponding piles.

\item The convex hulls of $\Yc_\cpc$-piles are disjoint.

\item The image of a $\Yc_\cpc$-pile is a $\Yc_\cpc$-pile.

\item The full preimage of a $\Yc_\cpc$-pile $\Ac$ consists of at most $d$ $\Yc_\cpc$-piles each of which is a pullback of $\Ac$.

\item The convex hull of every $\Yc_\cpc$-pile eventually maps onto a proper critical set from $\cpc^+$.

\end{enumerate}

\end{lem}

\begin{proof}
Left to the reader.
\end{proof}

Recall that
given two points $a, b\in\uc$ we denote by $(a,b)$ the arc such that the movement from $a$ to $b$ within this arc is
in the positive direction. Recall also that a \emph{hole} of a closed subset $A\subset \uc$ is a component of $\uc\sm A$.

\begin{dfn}\label{d:pos-ori1}
Given a closed set $A\subset\uc$, we say that the restriction $\si_d|_A$ 
is a \emph{positively oriented map} if for every hole $(x,y)$ of $A$ either
$\si_d(x)=\si_d(y)$ or the interval $(\si_d(x),\si_d(y))$ is also a hole of $\si_d(A)$.
\end{dfn}

If $\si_d|_A$ is positively oriented then it can be extended over the boundary of $\ch(A)$
to a map that sends the boundary of $\ch(A)$ to the boundary of
$\ch(\si_d(A))$ as a positively oriented composition of a monotone map and a covering map.


\begin{lem}\label{l:convex1}
If $\Ac$ is a $\Yc_\cpc$-pile and $\Vc$ is the set of vertices of $\bigcup \Ac$, then the map $\si_d|_\Vc$ is positively oriented.
In particular, $\si_d$ maps edges of $\ch(\bigcup \Ac)$ to edges of $\ch(\bigcup \si_d(\Ac))$. Moreover,
if $\ch(\bigcup \Ac)$ is not all-critical then its edges are non-critical and their union maps $k$-to-$1$ to edges of
$\ch(\bigcup \si_d(\Ac))$ for some $k\ge 1$.
\end{lem}

\begin{proof}
The claim is easy if $\ch(\bigcup \Ac)\subset \ol{U}$ where $U$ is a complementary component to $\cpc^+$.
Otherwise there are
sets $\Cc_i\subset \Ac$. Indeed, since $\Ac$ is not contained in the closure of a complementary component
then $\Ac$ is non-disjoint from some proper sets $\ch(\Cc_i)$; hence, by definition of a $\Yc_\cpc$-pile
all chords from $\Cc_i$ are in $\Ac$.

Denote by $D_1,$ $\dots,$ $D_k$ all complementary components
intersecting $\bigcup \Ac$. Then each $D_j$ has an edge of at least one such set $\ch(\Cc_i)$ in its boundary.
It follows that for each chord $\ell'$ in $\si_d(\Ac)$ there is at least one chord $\ell_j$ in each $D_j$
that maps to $\ell'$ (depending on whether the endpoints of $\ell'$ are images of critical sets there may be more
pullbacks of $\ell'$ in $D_j$). Moreover, since the order is preserved under $\si_d$ on $\bd(D_j)$ (except for the
collapse of the critical edges) then $\ch((\bigcup \Ac)\cap D_j)$ maps order preservingly
onto $\ch(\bigcup \si_d(\Ac))$.

To prove that this implies the desired statement, first assume that there is only one critical set $\Cc\in \cpc^+$
contained in $\Ac$. 
Then behind each edge $\ell$ of $\Cc$ there is a well-defined component $D_i$ of $\cdisk\sm \Cc$ bordering on $\ell$
that intersects $\bigcup \Ac$, and by the previous paragraph
the map on $\Vc\cap D_i$, and hence on $\Vc$, is positively oriented.

By induction, suppose now that the claim is proven in the case when $s$ critical sets $\Cc_i$ are contained in $\Ac$, and prove
it in the case when $s+1$ critical sets $\Cc_i$ are contained in $\Ac$. This we will also prove by induction, now over the number of edges
of a specially chosen set $\Cc\subset \Ac$.

Namely, find a critical set $\Cc$ that has an edge $\ell=\ol{ab}$ such that
the circle arc $I=(a, b)$ is disjoint from other critical sets $\Cc\subset \Ac$ ($\ell$ is an edge ``on the edge''). The induction is over the number $m$ of
edges of $\bigcup \Cc$. Indeed, suppose that $\Cc=\{\ell\}$. Then collapsing  $\ch(\ell\cup I)$ and applying the induction over $s$
we see that the resulting (after collapse) set $\Ac'$ intersected with $\uc$ maps onto its image positively oriented. Adding to this the chords of $\Ac$
with endpoints in $\ol{I}$, we see that the resulting map on $\Vc$ is positively oriented, too. This as the basis of induction over $m$, and the arguments
in the inductive step repeat (verbatim) the arguments in the base of induction case. This completes the proof of all the claims of the lemma  except for the last one.

To prove the last claim observe that if $\ch(\bigcup \Ac)$ is not all-critical but there is a critical edge $\ell$ of $\ch(\bigcup \Ac)$ then
$\ell$ must be an edge of a proper set $\Cc\in \cpc^+$ and, hence, cannot be an edge of $\ch(\bigcup \Ac)$ (sibling sets of $\ch(\bigcup \Ac)$
must be attached to all edges of $\Cc$ and those sibling sets are non-degenerate because $\ch(\bigcup \Ac)$ is not all-critical). The contradiction shows that
in the case at hand edges of $\ch(\bigcup \Ac)$ are not critical; hence, since $\si_d|_\Vc$ is positively oriented, then
the edges of $\ch(\bigcup \Ac)$ map $k$-to-$1$ to edges of $\ch(\bigcup \si_d(\Ac))$ for some $k\ge 1$.
\end{proof}

The next lemma is based on simple geometric considerations.

\begin{lem}\label{l:prelam}
The family of all edges of convex hulls of $\Yc_{\cpc}$-pile sets is a proper invariant prelamination.
Moreover, all such edges belong to the corresponding piles.
\end{lem}

\begin{proof}
Let $\Zc$ be a $\Yc_\cpc$-pile. By Lemmas \ref{l:yweb} and \ref{l:convex1} the family of all its pullbacks
consists of pairwise disjoint piles with pairwise disjoint convex hulls, and each pullback
pile maps onto $\Zc$ positively oriented. For each edge $\ell$ of the convex hull of a pullback pile of
$\bigcup \Zc$ we can choose its sibling edges in this pullback pile and among other pullback piles of $\Zc$.
By the last claim of
Lemma \ref{l:convex1} we can do this so that all these chords are pairwise disjoint, and by Lemma \ref{l:yweb} there are $d$ of them
(including $\ell$). Thus, the family of all edges of convex hulls of $\Yc_{\cpc}$-piles has the sibling property.
By Lemma \ref{l:convex1} this family is invariant. By Lemma \ref{l:basic1}(4) this family is a prelamination, and
by definition it is proper. The last claim of the lemma follows from Lemma \ref{l:basic1}(4).
\end{proof}

Observe that by Lemma \ref{l:prelam} all edges of convex hulls of $\Yc_{\cpc}$-pile sets are eventually mapped to edges of
\emph{proper} sets from $\cpc^+$ and never have images crossing edges of \emph{any} sets from $\cpc^+$.


\begin{thm}\label{t:side-lam}
Let $\cpc$ be a full critical collection. Then
the closure of the family of all edges of convex hulls of $\Yc_{\cpc}$-piles is a proper invariant lamination.
\end{thm}

\begin{proof}
Follows from Lemma \ref{l:prelam} and Lemma \ref{l:prop-clos}.
\end{proof}

We can now give the next definition.

\begin{dfn}\label{d:root-side} The proper invariant lamination from Theorem \ref{t:side-lam}
will be denoted by $\eYc_{\cpc}$.
\end{dfn}

A useful claim about $\eYc_{\cpc}$ is proven next.

\begin{lem}\label{l:ey-nocross}
No eventual image of a chord of $\eYc_{\cpc}$ can cross an edge of any set from $\cpc^+$ (thus,
$\eYc_{\cpc}$ is compatible with $\cpc^+$).
\end{lem}

\begin{proof}
The lemma is proven in Lemma \ref{l:prelam} in the case when we consider an edge of the convex hull
of a pile set. On the other hand, if a chord of $\eYc_{\cpc}$ is the limit of convex hulls of $\Yc_{\cpc}$-pile sets
then, since their edges belong to the corresponding $\Yc_{\cpc}$-piles, the claim easily follows by continuity.
\end{proof}

Since $\eYc_{\cpc}$ is proper, then it defines a q-lamination.

\begin{dfn}
The q-lamination $\lam^{pr}_{\cpc}$ defined by $\eYc_{\cpc}$ is called the \emph{preroot (q-lamination) of $\cpc$}.
The associated laminational equivalence is denoted by $\sim^{pr}_\cpc$.
\end{dfn}

It turns out that $\lam^{pr}_\cpc$ may have leaves that do not belong to $\eYc_{\cpc}$. Moreover, it turns out that these leaves
may even be incompatible with $\cpc^+$. In what follows we will study the compatibility of $\lam^{pr}_\cpc$ and $\cpc^+$.


\begin{lem}\label{l:prer-cross}
Suppose that $\ell$ is a leaf of $\lam^{pr}_\cpc$ incompatible with $\cpc^+$. Then $\ell$ is a common
edge of a finite gap $A$ of $\lam^{pr}_\cpc$ and an infinite (pre)hyperbolic gap $U$ of $\lam^{pr}_\cpc$ which
eventually maps to a periodic non-degenerate edge of a periodic hyperbolic gap. In particular, if $\cpc^+$ is a Side full critical collection,
then $\lam^{pr}_\cpc=\eYc_\cpc$.
\end{lem}

\begin{proof}
 By Lemma \ref{l:ey-nocross}
it follows that $\ell\notin \eYc_{\cpc}$. Hence $\ell$ is an edge of the convex hull of a $\sim^{pr}_\cpc$-class which was added
as we constructed $\lam^{pr}_\cpc$ by taking convex hulls of $\approx_{\eYc_\cpc}$ sets. Hence the $\sim^{pr}_\cpc$-class of $\ell$
must consist of at least three points, the corresponding convex hull $A$ is a finite gap, $\ell$ is an edge of $A$, and $\ell$ is
not approached from the outside of $A$ by leaves of $\eYc_\cpc$. We conclude that $A$ is a finite gap attached to an infinite gap $U$
of $\eYc_\cpc$.

Observe that by Theorem \ref{t:kiwan} $U$ is (pre)periodic. Hence by Lemma \ref{l:edges-of-gaps} edges of $U$ are
(pre)periodic or (pre0critical. We claim that if we map $A$ forward then it must be eventually mapped to a finite periodic gap. Indeed, otherwise
by the above  the only thing
that can happen is that $A$ will be eventually mapped to a critical leaf which by construction implies that $A$ in fact is a gap from $\eYc_\cpc$
and, hence, $\ell$ must be compatible with $\cpc^+$, a contradiction with the assumption. Furthermore, the same argument implies that
in the absence of hyperbolic gaps the above described phenomenon (when a finite gap $A$ of $\lam^{pr}_\cpc$ has an edge that does not belong
to $\eYc_\cpc$) is impossible. This completes the proof.
\end{proof}

The next lemma studies pullbacks of leaves for laminations compatible with $\cpc^+$. Since elements of
$\cpc^+$ are convex polygons (see Definition \ref{d:poext2}), we can simply talk about \emph{complementary}
components of $\cpc^+$ which are the same as previously used \emph{partially critical complementary} components.

\begin{lem}\label{l:pull}
Suppose that $\cpc$ is a full critical collection and
a $\si_d$-invariant lamination $\lam$ is compatible with $\cpc^+$. Let
$\ell\in \lam$, and let $U$ be a component of $\cdisk\sm \bigcup \cpc^+$.
Then there exists a leaf $\ell'\in \lam$ with $\si_d(\ell')=\ell$ and $\ell'\subset \ol{U}$.
\end{lem}

\begin{proof}
By definition there exists a sibling collection of leaves $\ell_1,$
$\dots,$ $\ell_d$ in $\lam$ each of which maps to $\ell=\ol{a b}$. Each preimage of $a$ and $b$
shows once among the endpoints of leaves $\{\ell_i\}$. Hence if $a$ is not the image of any critical edge
of $U$, then there exists a leaf $\ell_j$ contained in $\ol{U}$. Now, suppose that $\hell=\ol{xy}$ is a critical edge
of $U$ with $\si_d(\hell)=a$. The leaf $\hell$ cuts $\cdisk$ in two \emph{open} components $A$ and $B$ with, say,
$U\subset B$. In both $A$ and $B$ the number of
point-preimages of $b$ is by one greater that the number of preimages of $a$. Hence we may assume that a leaf
$\ell_s=\ol{xu}\in \lam$ is contained in $A$ and a leaf $\ell_t=\ol{yv}\in \lam$ is contained in $B$ which implies that $\ell_t\subset \ol{U}$
as desired.
\end{proof}

Let $U_1, \dots, U_d$ be complementary domains of $\cpc^+$.
Then the boundary of each $U_i$ is an alternating concatenation of circle arcs and critical leaves.
For $y\in \uc$, either exactly one point of $\bd(U_i)$ maps to $y$ under $\si_d$, or there are two such
points (which are the endpoints of one critical edge in the boundary of $U_i$).
Finally, the endpoints of each critical chord from $\cpc^+$ that does not come from a portal
belong to the same $\sim^{pr}_\cpc$-class. Let $\pi_i:\partial U_i\to \uc$ be the restriction of $\si_d$ to $\partial U_i\cap \uc$
and call it \emph{the induced map}.

We now need the results from the Appendix describing properties of periodic gaps
with first return map of degree one (e.g., we need Definition \ref{d:deg1} where \emph{(pre-)caterpillar}
and \emph{(pre-)Siegel} (periodic) gaps are defined). Thus, the reader is advised to read the Appendix before continuing to read
here.

\begin{lem}\label{l:noc-nod}
If $\cpc$ is a Side full critical collection then the only infinite gaps that the q-lamination $\lam^{pr}_\cpc$ may have are
(pre-)Siegel gaps.
\end{lem}

\begin{proof}
If $\lam^{pr}_\cpc$ has a (periodic) caterpillar gap then by Lemma \ref{l:crit-must} it must have a critical leaf with a periodic
endpoint, a contradiction (in fact since $\lam^{pr}_\cpc$ is compatible with $\cpc$ then $\lam^{pr}_\cpc$ cannot be compatible with a critical chord
with a periodic endpoint as the latter must have both endpoints coming from the set of the endpoints of chords from $\cpc$ which has
no periodic points). Hence $\lam^{pr}_\cpc$ has no (pre-)caterpillar gaps. Now, if $\lam^{pr}_\cpc$ has a periodic gap $U$ with first return map of
degree greater than one then we may assume that $U$ is critical which implies that at least one set from $\cpc$ is contained in $U$,
a contradiction to the definition of $\lam^{pr}_\cpc$. Hence the only possibility for an infinite gap of $\lam^{pr}_\cpc$ is to be
(pre-)Siegel as claimed.
\end{proof}


We are ready to prove the last theorem of this section.

\begin{thm}\label{t:1class} Suppose that $\cpc$ is a Side full critical collection and $\ell=\ol{ab}$ is a chord so that
no two images of $\ell$ cross and no image of $\ell$ crosses an edge of $\cpc^+$.
Then the endpoints of $\ell$ are in the same $\sim^{pr}_\cpc$-class.
\end{thm}

\begin{proof}Put $\si_d^n(\ell)=\ell_n=\ol{a_nb_n}$; choose for each $n$
a complementary domain $U_n$ of $\cpc^+$ with $\ell_n\subset \ol{U_n}$.
By way of contradiction suppose that $a$ and $b$ are in distinct $\sim^{pr}_\cpc$-classes. We claim that then
$a_n$ and $b_n$ belong to distinct $\sim^{pr}_\cpc$-classes for each $n$. Indeed, suppose that $a_{n+1}$ and
$b_{n+1}$ belong to the same class. Observe that $a_{n+1}$ and $b_{n+1}$ belong to $\ol{U_n}$, and denote by $A$ the intersection of the convex hull of
the class containing $a_{n+1}$ and $b_{n+1}$ with $\ol{U_{n+1}}$. By Lemma \ref{l:dynama}, by Lemma \ref{l:prer-cross},
and by definitions the pullback of $A$ into
$\ol{U_n}$ is a convex set whose intersection with $\uc$ is contained in one class. Hence $a_n$ and $b_n$ belong to one class, and
the argument can be repeated. Thus, under our assumption we have that $a_n$ and $b_n$ belong to distinct $\sim^{pr}_\cpc$-classes for each $n$.
Consider cases.

(1) Suppose that $\ell$ is contained in a gap $W$ of $\lam^{pr}_\cpc$. Then $W$ is infinite.
By Lemma \ref{l:noc-nod} $W$ must be pre-Siegel. Hence two forward images
of $\ell$ will cross each other, a contradiction.

(2) The remaining case is when for each $n$ there exists a chord $\od_n$ which
crosses $\ell_n$ and maps in, say, $k_n$ steps to an edge of $\cpc^+$. Since no image
of $\ell_n$ crosses an edge of $\cpc^+$, there exists $m_n<k_n$ so that endpoints of
$\si^{m_n}_d(\od_n)$ and $\si^{m_n}_d(\ell_n)$  are connected by a critical leaf from $\cpc^+$.
However, since $\cpc^+$ is Side (i.e., finite without periodic vertices), there exists an $N$ such that for
all $n\ge N$, $\ell_n\cap\, \cpc^+=\0$. Hence $a_n$ and $b_n$ cannot be contained in distinct
$\sim_{pr}$ classes for all $n$, a contradiction.
\end{proof}

Observe that in the situation of Theorem \ref{t:1class} the assumption that $\cpc$ is Side is essential.
If there are some portals, then it may well happen that several (more than one) portal
are inside infinite hyperbolic gaps of $\lam^{pr}_\cpc$ and there are (periodic) leaves inside those gaps that satisfy all the
conditions of Theorem \ref{t:1class} but, evidently, are not leaves of $\lam^{pr}_\cpc$.

\section{Introduction to Side laminations}\label{s:sidelami}

In this section we relate the just introduced concepts and laminations. Let us recall
several definitions.

\begin{dfn}[Dendritic laminations]\label{d:den}
A q-lamination $\lam$ is \emph{dendritic} if it has no infinite gaps. A laminational equivalence relation $\sim$ is \emph{dendritic}
if $\lam_\sim$ is dendritic. If $\sim$ is invariant then a topological polynomial $f_\sim$ is \emph{dendritic} if $\lam_\sim$ is dendritic.
\end{dfn}

The next lemma follows.

\begin{lem}\cite[Lemma 2.8]{bopt19}\label{l:2.18}
Leaves of dendritic q-la\-mi\-na\-tions are not isolated.
\end{lem}

By Theorem \ref{t:kiwan} any infinite gap is (pre)periodic.
In this section, in addition to dendritic laminations, we consider invariant q-laminations that are non-dendritic
by virtue of having periodic \emph{Siegel} gaps. Thus, we consider only q-laminations (and other related objects such as topological
polynomials etc) that have no hyperbolic infinite gaps (equivalently,  all infinite periodic gaps are Siegel).

\begin{dfn}[Siegel-dendritic laminations]\label{d:side}
A q-la\-mi\-na\-tion is said to be \emph{Siegel} if it has a periodic Siegel gap.
A q-lamination all of whose infinite periodic gaps are Siegel is said to be a \emph{\textbf{Si}egel-\textbf{de}ndritic} lamination, or just a \emph{Side} lamination.
The same terminology applies to laminational equivalence relations, topological polynomials, etc.
\end{dfn}

The situation is a bit more involved if we consider \emph{all} laminations, not only q-laminations.
We need the following definition.


\begin{dfn}[Side critical collections]\label{d:sidepo}
\emph{Side}  critical collections (critical portraits) are those whose critical chords have no periodic endpoints.
\end{dfn}

The next lemma is easy.

\begin{lem}\label{l:onlys}
A Side q-lamination $\lam$ can only be compatible with Side critical portraits.
\end{lem}

\begin{proof}
The claim follows from the fact that critical sets of $\lam$ have no periodic vertices.
\end{proof}

The next definition applies to \emph{all} laminations, not only q-laminations.

\begin{dfn}\label{d:uni-side}
An invariant lamination is \emph{Side} if it is only compatible with Side critical portraits
(equivalently, if any critical chord with a periodic endpoint is incompatible with $\lam$).
\end{dfn}


\subsection{Capture dynamics}


We need Definition \ref{d:capture}.

\begin{dfn}[Capture dynamics]\label{d:capture}
A lamination $\lam$ is a \emph{capture} lamination if it has a non-periodic infinite gap $U$ such that the restriction
$\si_d|_U$ is of degree greater than one.
\end{dfn}

Capture dynamics needs at least two critical sets to be implemented. Indeed, the gap $U$ from Definition \ref{d:capture}
is critical; in addition, the cycle of infinite gaps into which $U$ eventually maps contains at least one gap with some criticality,
either on its boundary (in the caterpillar or Siegel cases) or inside (in the hyperbolic case).
Thus, in the quadratic case capture dynamics is impossible.
In the cubic case capture dynamics is to some extent well-defined. Indeed,
let $\lam$ be a cubic lamination and let $V$ be a periodic eventual image of $U$. Then there are the following three cases.

\begin{enumerate}

\item If $V$ is hyperbolic, then $\lam$ is called a \emph{hyperbolic capture} lamination.

\item If $V$ is a Siegel gap, then $\lam$ is called a \emph{Siegel capture} lamination.

\item If $V$ is a caterpillar gap, then $\lam$ is called a \emph{caterpillar capture} lamination.

\end{enumerate}

Laminations that do not exhibit capture dynamics are said to be \emph{non-capture}.
The same terminology applies to topological polynomials and all other related objects.

\subsection{Side critical portraits}

We need a few known results.

\begin{lem}\cite[Lemma 2.5]{bopt16}\label{l:limgap1}
Let $\lam=\lim \lam_i$ where $\lam_i$ are $\si_d$-invariant q-laminations.
Let $G$ be an $n$-periodic gap of $\lam$.
Then, for $i$ sufficiently large, $\lam_i$ have $n$-periodic gaps $G_i$
that converge to $G$.
\end{lem}

Recall the following fundamental theorem of Jan Kiwi (the second part of the theorem  easily follows from the first one).

\begin{thm}\cite[Theorem 4.4]{kiw02}\label{t:kiwan}
A wandering non-precritical gap of a $\si_d$-invariant lamination can have at most $d$ vertices.
In particular, infinite gaps of laminations are (pre)periodic.
\end{thm}

Observe that in the cubic case there exist q-laminations with wandering triangles \cite{bo04, bo09}.

One can now describe properties of Side critical portraits.

\begin{lem}\label{l:sicripo}
The following claims hold.

\begin{enumerate}

\item If $\cpc$ is a Side full critical collection, then all critical sets of $\lam^{pr}_\cpc$ are finite.

\item A q-lamination has only finite critical sets iff it is a non-capture Side q-lamination.

\item A full critical collection $\cpc$ is Side iff it is compatible with a Side q-lamination.

\end{enumerate}

\end{lem}

\begin{proof}
(1) By Theorem \ref{t:side-lam} $\lam^{pr}_{\cpc}$ is a
q-lamination. Its critical sets contain elements of $\cpc^+$, hence they are classes under the equivalence relation
associated with $\lam^{pr}_{\cpc}$. It follows that critical sets of $\lam^{pr}_{\cpc}$ are finite.


(2) Let $\lam$ be a non-capture Side q-lamination.
By Theorem \ref{t:kiwan}, an infinite critical set of a q-lamination is either a preperiodic infinite gap, or a periodic hyperbolic gap,
and both possibilities are ruled out for $\lam$ by the assumptions. Hence $\lam$ has only finite critical sets.
The converse claim is immediate. In particular, if $\cpc$ is a Side full critical collection,
then $\lam^{pr}_{\cpc}$ is a non-capture Side q-lamination.

(3) Let $\cpc$ be a Side full critical collection. It is compatible with $\lam^{pr}_{\cpc}$, by (1) all critical sets of $\lam^{pr}_{\cpc}$ are
finite, and by (2) $\lam^{pr}_{\cpc}$ is a non-capture Side q-lamination. On the other hand, let $\lam$ be a Side q-lamination compatible with
some full critical collection $\cpc$. Then by definition critical sets of $\lam$ have no periodic vertices, hence
$\cpc$ contains no critical chords with a periodic endpoint, i.e. is Side.
\end{proof}

It turns out that there is a connection between Side critical portraits and proper laminations.
The next definition is very useful when studying laminations.

\begin{dfn}\label{d:cone}
A family $\Xc$ of chords $\ol{ab}$ sharing the same endpoint $a$
is said to be a \emph{cone}. The point $a$ is
called the \emph{vertex} of the cone; the set $\uc\cap \bigcup \Xc$ is called
the \emph{basis} of the cone and is denoted by $\Xc'$. A cone is said
to be \emph{infinite} if it consists of infinitely many chords.
\end{dfn}

It turns out that cones of invariant laminations have specific properties. The next claim is based \cite[Lemma 2.4]{bopt16}.

\begin{lem}\rm{(cf.} \cite[Lemma 2.4]{bopt16}\rm{)}\label{l:inficon1}
\emph{Let $\Xc$ be an infinite cone of an invariant lamination $\lam$. Then the following claims hold.}

\begin{enumerate}

\item \emph{$\Xc$ has a (pre)periodic vertex and
all leaves in $\Xc$ are either (pre)critical or (pre)periodic.}

\item \emph{$\lam$ has a set $Y$ which is a
leaf or a concatenation of two leaves, the points $a$ and $b$ are the endpoints of $Y$, $\si_d(a)=\si_d(b)$,
and $b=\si_d^k(b)$ for  some integer $k>0$.}

\item \emph{$\lam$ is not compatible with any Side critical portrait.}

\end{enumerate}

\end{lem}

\begin{proof}
(1) By Theorem \ref{t:kiwan} $\Xc$ is eventually mapped to an infinite cone with a periodic vertex. Then the claim
follows from Lemma 2.4 of \cite{bopt16}.

(2) Without loss of generality assume that the vertex $v$ of $\Xc$ is fixed.
Evidently, if there is a critical leaf coming out of $v$ then we are done. Otherwise notice that in an invariant
lamination periodic points can only be connected with a leaf if they are of the same period. It follows from the above then
that there are points $a$ and $b$ such that $\ol{va}\in \Xc$, $\ol{vb}\in \Xc$ and $\si_d(a)=\si_d(b)=b$.

(3) By way of contradiction assume that $\lam$ is compatible with a Side critical portrait $\Po$. Then the entire set $Y$ must be
contained in the closure of a partially critical complementary domain to $\bigcup \Po$. The properties of $Y$ now imply that the periodic
point $b$ must be an endpoint of a critical chord from $\Po$, a contradiction.
\end{proof}

Lemma \ref{l:propeside} relies on Lemma \ref{l:inficon1}.

\begin{lem}\label{l:propeside}
An invariant lamination $\lam$ is proper if and only if it is compatible with a Side critical portrait.
\end{lem}

\begin{proof}
Suppose that $\lam$ is compatible with a Side critical portrait. By Lemma \ref{l:inficon1} $\lam$ cannot have
an infinite cone. We claim that then $\lam$ is proper. Indeed, suppose that $\lam$ is not proper. Then
by definition there is a leaf $\ol{xv}$ connecting a periodic point $v$ and a non-periodic point $x$. Pulling it back
we can find a leaf $\ol{yv}$ such that $\si_d^k(y)=x$; clearly, $y\ne x$. Repeating it infinitely many times we will
discover an infinite cone with periodic vertex $v$, a contradiction.

Suppose that $\lam$ is proper. Consider critical sets of $\lam$. If none of them has periodic vertices then any critical portrait that consists
of chords inserted into critical sets of $\lam$ will be a desired Side critical portrait compatible with $\lam$. Suppose now that $\lam$ has
a critical set $C$ with a periodic vertex. Since $\lam$ is proper, then $C$ is an infinite gap; for the same reason no image of $G$ is a caterpillar gap
(recall that by Lemma \ref{l:crit-must} a caterpillar gap has a critical edge with a periodic endpoint). Hence $C$ is (pre-)Siegel or (pre-)hyperbolic.
Evidently, in these cases one can insert in $C$ the corresponding number of critical chords with non-periodic endpoints. This completes the proof.
\end{proof}

We need the following definition.

\begin{dfn}\label{d:sqsubset}
Let $A$ be a gap/leaf and $B$ be a gap. If $A\subset B$ and some edges of $A$ intersect the interior of $B$,
then we denote it by $A\sqsubset B$. Observe that if $\ell$ is an 
edge of $B$ then $\ell\not\sqsubset B$. If $A\sqsubset B$ then we say that $A$ \emph{$\sqsubset$-contained} in $B$.
The symbol $\sqsupset$ is used analogously.
\end{dfn}

We can now describe Side laminations in a more dynamical fashion.

\begin{lem}\label{l:side2}
A lamination $\lam$ is Side if and only if it is proper and has no periodic hyperbolic gaps.
\end{lem}

\begin{proof}
If $\lam$ is Side then by Lemma \ref{l:propeside} it is proper. On the other hand, it cannot have a periodic hyperbolic gap
because otherwise one can find a critical portrait compatible with $\lam$ which includes a critical chord with a  periodic endpoint.

On other hand, let $\hlam$ be a proper lamination with no periodic hyperbolic gaps.
By way of contradiction suppose that $\hlam$ is compatible with a critical portrait $\Po$ which includes
a critical chord $\ell$ with a periodic endpoint.
Since $\hlam$ is proper, $\ell\notin \hlam$. Thus $\ell\sqsubset U$ where $U$ is a gap. Since $\ell$ has a periodic endpoint
and $\lam$ has no critical leaves with a periodic endpoint, then $U$ is a hyperbolic periodic gap, a contradiction.
\end{proof}

\subsection{Perfect laminations}\label{ss:perfect}

Every lamination $\lam$ has a maximal perfect subset $\lam^p$ called the \emph{perfect part} of $\lam$.

\begin{lem}\cite[Lemma 3.12]{bopt20}\label{l:3.12}
$\lam^p$ is an invariant perfect lamination.
For every $\ell\in \lam^p$ and every neighborhood $U$ of $\ell$ in the topology of the plane, there
exist uncountably many leaves of $\lam^p$ in $U$.
\end{lem}

By Lemma \ref{l:crit-must} a perfect lamination cannot have infinite periodic gaps of degree one (in particular, there are no periodic Siegel gaps). Moreover, if two
infinite gaps are non-disjoint, there will have to be isolated leaves, so this is ruled out too. However, periodic hyperbolic gaps in general are
possible. For example, if a quadratic polynomial has a cycle of hyperbolic domains of \emph{primitive type}
(see Definition~\ref{d:satellite}), then the associated lamination is perfect.

The structure of a non-empty perfect q-lamination $\lam$ can be described as follows.
Consider a monotone quotient map of the closed unit disk that collapses not only leaves and
finite gaps but also infinite gaps of $\lam$. This map gives rise to a dendritic quotient space
which formally speaking is not a topological Julia set (it has points of infinite order) but has very similar properties. We shall call such spaces
\emph{generalized dendritic topological Julia sets} and the induced maps \emph{generalized topological polynomials}.
Such maps were studied in \cite{bl02} where the following lemma was proved.

\begin{lem}\cite[Theorem C]{bl02}\label{l:nowanco}
If $f:J\to J$ is a generalized topological polynomial on a generalized dendritic topological Julia set
then for every continuum $T\subset J$ there are two numbers $m<n$ such that $f^m(T)\cap f^n(T)\ne \0$
(one can say that $f$ has no wandering continua).
\end{lem}

Lemma \ref{l:nowanco} is used in the proof of Lemma \ref{l:condense}.

\begin{lem}\cite[Lemma 3.25]{bopt20}\label{l:condense}
Suppose that $J$ is a generalized topological Julia set generated by a perfect q-lamination
$\lam_\sim$. Then the following holds.

\begin{enumerate}

\item Each subcontinuum of $J$ contains a $($pre$)$periodic
    non-$($pre$)$critical point.

\item Each subcontinuum of $J$ contains a $($pre$)$critical
    point.

\item Each leaf of $\lam_\sim$ can be approximated by
    $($pre$)$periodic leaves that never map to a critical set of
    $\lam_\sim$.

\end{enumerate}

\end{lem}

\subsection{Perfect-Siegel laminations}

Another class of laminations is formed by \emph{Siegel laminations}. To define it, we need a technical concept
which we now introduce. Let $U$ be an infinite gap of a lamination $\lam$.
Moreover, suppose that either $U$ is a (pre)periodic gap eventually mapped to a periodic Siegel gap,
or that $\lam$ is a q-lamination and $U$ is a (pre)periodic infinite gap.
There may exist finite gaps attached to $U$ along some edges of $U$. Moreover, there may be other finite gaps
attached to these finite gaps along their edges, etc. Yet, there are only finitely many such planar concatenations
of finite gaps. Indeed, if $\lam$ is a q-lamination this follows from the fact that finite gaps of q-laminations are disjoint.
On the other hand, if $U$ is (pre)Siegel then the planar concatenation of finite gaps described above
is necessarily wandering and, hence, by Theorem \ref{t:kiwan} finite.



It is easy to see that every finite gap that is attached to a periodic Siegel gap eventually maps to a point on the boundary of this Siegel gap.

\begin{dfn}\label{d:siegel-part}
Let $\lam$ be an invariant lamination. Consider the grand orbits of all edges of all its periodic Siegel gaps. This is a collection
of leaves of $\lam$. It is easy to see that this is a prelamination called the \emph{non-closed Siegel part}
of a lamination $\lam$. The non-closed Siegel part of $\lam$ can be closed (i.e., we can add to it all limit chords of
its leaves) to form a sublamination of $\lam$ called the \emph{Siegel part} of $\lam$. If $\lam=\lam^{Sie}$ then $\lam$ is called a \emph{Siegel lamination}.
\end{dfn}

Combining the perfect and Siegel parts of a lamination we obtain the
\emph{perfect-Siegel part} of a given lamination.

\begin{dfn}[Perfect-Siegel part]\label{d:pspart}\index{perfect-Siegel part}
The union $\lam^{pS}$ of the perfect and the Siegel parts of an
invariant lamination $\lam$ is called the \emph{perfect-Siegel
part} of $\lam$. If a lamination coincides with its perfect-Siegel part,
it is called a \emph{perfect-Siegel}.
\end{dfn}

We are ready to prove the next theorem.

\begin{lem}\label{l:side-ps}
The following claims hold.

\begin{enumerate}

\item Every Side q-lamination $\lam$ is perfect-Siegel.

\item Let $\cpc$ be a full Side critical collection. Then $\lam^{pr}_{\cpc}$ is perfect-Siegel with finite critical sets.

\end{enumerate}

\end{lem}

\begin{proof}
(1) Assume that $\lam^{pS}\ne \lam$. Then, $\lam$ has
a leaf $\ell$ with a small neighborhood and countably many leaves in it. Moreover, all these leaves
never map to edges of periodic Siegel
gaps. Choose an isolated leaf $\hell$ among them. Then $\hell$ is an edge of an infinite gap $W$ of $\lam$;
by Theorem \ref{t:kiwan} $W$
has eventual forward image $U$ that is periodic and must be Siegel (because $\lam$ is Side), a contradiction.

(2) By Lemma \ref{l:sicripo}(1) and the assumption that $\cpc$ is Side, all critical sets of $\lam^{pr}_{\cpc}$ are finite.
By Lemma \ref{l:sicripo}(2) $\lam^{pr}_{\cpc}$ is a Side non-capture lamination as claimed.
\end{proof}

The following is a particular case of Theorem 3.57 of \cite{bopt20}.

\begin{thm}\label{t:noesli}
If $\lam_1$ and $\lam_2$ are invariant laminations that both \textbf{contain} the same
critical portrait $\Po$, then $\lam^{pS}_1=\lam^{pS}_2$.
\end{thm}

\subsection{Root laminations: the Side case}\label{ss:root1}

Let us now give an important definition.

\begin{dfn}\label{d:root1}
Suppose that $\cpc$ is a Side full critical collection. A q-lamination $\lam$
is said to be a \emph{root} q-lamination (with the associated \emph{root} laminational equivalence $\sim$) of $\cpc$
if given a (proper) critical set $A$ from $\cpc^+$ all vertices of $A$ belong to the same $\sim$-class.
\end{dfn}

The next theorem is the main theorem of this section.

\begin{thm}\label{t:side-full}
If $\cpc$ is a Side full critical collection and $\lam$ is a Side non-capture q-lamination compatible with $\cpc^+$
then $\lam=\lam^{pr}_{\cpc}$ is a unique root lamination compatible with $\cpc$. A q-lamination $\lam$ is a root
lamination of a Side full critical collection $\cpc$
if and only if it is a non-capture q-lamination compatible with $\cpc^+$.
\end{thm}

Thus, Side non-capture q-laminations are completely defined
by any full critical collection compatible with them: if two such laminations share a full critical collection
(for instance, if they share a critical portrait) then they coincide.

\begin{proof}
Critical sets of $\cpc^+$ are all-critical. Suppose that $C\in \cpc^+$ and show that $C$ is contained in a gap-leaf
of $\lam$. Observe that since $\lam$ is non-capture, edges of $C$ cannot be contained in infinite gaps of $\lam$. Hence
all edges of $C$ are contained in finite gap-leaves. Since distinct gap-leaves of $\lam$ are pairwise disjoint
(they are, after all, convex hulls of finite classes of equivalence), it follows that $C$ is contained in a gap-leaf
of $\lam$ as claimed.

Hence pullbacks of all critical sets from $\cpc^+$ are contained in the pullbacks of the corresponding gap-leaves of $\lam$.
In particular, these laminations are compatible.
It follows that all chords that are limits of pullbacks of sets from $\cpc^+$ are leaves of both $\lam$ and $\lam^{pr}_{\cpc}$.
Now, take a gap-leaf $B$ of $\lam^{pr}_\cpc$ and its edge $\ell$ which is not the limit of pullbacks of sets from $\cpc^+$.
Then $\ell$ is an edge of an infinite gap $S$ of $\lam^{pr}_\cpc$.

By Lemma \ref{l:side-ps}, $S$ is either itself a periodic Siegel gap, or maps onto such
a gap after some steps in the one-to-one fashion. Hence a chord $\ell\sqsubset S$ maps so that two of its images
cross, and so it cannot be a leaf of any lamination.
Thus, $B$ is a gap-leaf of $\lam$ and so $\lam^{pr}_\cpc\subset \lam$. Moreover, the same argument
shows that if $\ell\in \lam\sm \lam^{pr}_{\cpc}$ then there exists
a finite gap $G$ of $\lam^{pr}_\cpc$ with $\ell\sqsubset G$. However this is impossible because by the above
all edges of $G$ are leaves of $\lam$. Thus, $\lam=\lam^{pr}_\cpc$ as desired.

Finally, let us prove that a q-lamination $\lam$ is a root lamination of a Side full critical collection $\cpc$
if and only if it is a non-capture q-lamination compatible with $\cpc^+$.
Indeed, if $\lam$ is a root lamination of $\cpc$ then it is non-capture compatible with $\cpc^+$. On the other hand,
if $\lam$ is a non-capture Side q-lamination then no critical chord of $\cpc$ can be contained in an infinite gap of $\lam$ unless
it is an edge of such a gap. Hence all vertices of a critical set from $\cpc^+$ belong to the same $\sim$-class.

\end{proof}

\section{Cycles of gaps of preroot laminations}\label{s:hypegaps}



In the construction of $\lam^{pr}_{\cpc}$ we pull back \emph{only proper sets} of $\cpc^+$; yet, we use \emph{all sets} of $\cpc^+$ when we define the pullbacks.
Hence portals of $\cpc^+$ are compatible with $\eYc_{\cpc}$ (Definition \ref{d:root-side}). However, the same claim does not hold for
the \emph{q-lamination} $\lam^{pr}_{\cpc}$.

\begin{ex}\label{e:incompat}
Let $U$ be an invariant quadratic gap of $\lam^{pr}_{\cpc}$ with an invariant finite gap $G$ attached to $U$ along
an invariant edge $\ell$ of $U$. Clearly, all edges of $G$ are invariant and $\ell\notin \eYc_{\cpc}$ by definition.
Thus, $\ell$ has been added to $\eYc_{\cpc}$ when the q-lamination associated with $\eYc_{\cpc}$ was constructed.
Hence all other edges of $G$ must be leaves of $\eYc_{\cpc}$ as otherwise the construction of $\lam^{pr}_\cpc$ as the q-lamination
associated with $\eYc_{\cpc}$ would not have yielded $G$ as a gap of $\lam^{pr}_\cpc$.
It follows that at other edges of $G$ no gaps of $\lam^{pr}_\cpc$ are attached to $G$.

Clearly, there exists a unique edge $\ell'$ of $U$ with $\si_d(\ell')=\si_d(\ell)$, and a finite gap $G'$ attached to $U$
along $\ell'$ such that $\si_d(G')=G$. The chord $\ell'$ is a \emph{sibling (chord)} of $\ell$, and $G'$ can be called a \emph{sibling (gap)} of $G$. It may well happen that
there is a portal $\{\ol{xx'}\}$ where $x\in G\sm \ell$ is a vertex of $G$ and $x'$ is the associated to $x$ sibling vertex of $G'$. In that case $\ol{xx'}$ is
not compatible with $U$ because it crosses $\ell$.

However if we erase $\ell$ and its entire grand orbit, we will get a slightly extended version
$\widetilde{U}$ of $U$ which, by the above, is a gap of $\eYc_{\cpc}$, and $\widetilde{U}$ \emph{is} compatible with $\ol{xx'}$. The gap $\widetilde{U}$ is then called an
\emph{extension} of $U$. Observe that if we remove the grand orbit of $\ell$ from $\lam^{pr}_{\cpc}$ we obtain a \emph{proper} lamination that is not
a q-lamination and is strictly contained in $\lam^{pr}_{\cpc}$ but generates $\lam^{pr}_{\cpc}$ (see Definition \ref{d:gen-prop}). \hfill $\Box$
\end{ex}

Example \ref{e:incompat} is a particular case of an important general phenomenon. 
Namely, suppose that a proper lamination $\lam$ is compatible with a $\cpc^+$ where $\cpc$ is a full critical collection. Denote by $\sim$ the laminational equivalence relation generated by $\lam$ and by $\hlam$ the associated q-lamination. Then $\hlam$ may be incompatible with $\cpc^+$ because of the above described phenomenon: some extra-leaves of
$\hlam$ added to $\lam$ to complete convex hulls of $\sim$-classes may cross chords from $\cpc^+$. However this incompatibility
is not essential in the following sense: chords of $\cpc^+$ intersect at most two convex hulls of $\sim$-classes.

On the positive side, periodic hyperbolic gaps of $\lam^{pr}_{\cpc}$ have nice properties.

\begin{dfn}\label{d:satellite}
Let $U$ be a periodic hyperbolic gap of a lamination $\lam$. If there exists a finite gap $G$ such that $U$ is attached to $G$
along an edge and the period of $U$ is greater than that of $G$, then $U$ is said to be of \emph{satellite type}.
Otherwise $U$ is said to be of \emph{primitive type.}
\end{dfn}

A periodic hyperbolic Fatou gap $U$ with an attached to it finite periodic gap of the same period as $U$ 
is of primitive type. 

Recall that in Definition \ref{d:poext2} we defined a number of objects describing critical collections.
Given a full critical collection $\cpc$, we denote by
$\{\Kc_i(\cpc)=\Kc_i\}$ the collections of critical chords forming components of $\bigcup \cpc$ and call these components
(i.e., sets $\bigcup \Kc_i$) the
\emph{critical chordial continua (of $\cpc$)}.
For each $i$ we then denote by $\Cc_i(\cpc)=\Cc_i$ the family of edges of the convex hull $\ch(\bigcup \Kc_i)$ of the set $\bigcup \Kc_i$.

\begin{lem}\label{l:primiti}
Distinct periodic hyperbolic gaps of $\lam^{pr}_{\cpc}$ cannot share edges with one finite gap or with each other. Thus,
all periodic hyperbolic gaps of $\lam^{pr}_{\cpc}$ are of primitive type.
\end{lem}

\begin{proof}
The construction implies that all Fatou gaps of $\eYc_{\cpc}$ have their edges as either edges of
pullbacks of proper sets $\Cc_i$, or limits of such pullbacks. Thus, if $U$ and $\widehat{U}$ are periodic hyperbolic gaps
of $\lam^{pr}_{\cpc}$ sharing an edge $\ell$, or sharing edges $\ell$ and $\hell$ with
a periodic finite gap $G$ then $\ell,$ $\hell$ and all chords from their grand orbits
do not belong to $\eYc_{\cpc}$ because they are isolated and are not pullbacks of proper critical set.
\end{proof}

\begin{dfn}\label{d:tune-portal}
Let $\Hc_\cpc$ be the collection of all gaps of $\eYc_\cpc$ that eventually map to periodic hyperbolic
gaps of $\eYc_\cpc$.
\end{dfn}

We will need the following general claim.

\begin{lem}\label{l:allc}
Let $\lam$ be an invariant q-lamination. A finite gap $G$ of $\lam$ with a critical edge must be all-critical.
\end{lem}

\begin{proof}
If $G$ is not all-critical then there exists a critical edge $\ol{ab}$ of $G$ and a non-critical edge $\ol{bc}$ of $G$.
Then by \cite[Lemma 3.8(2)]{bmov13} there must exist a sibling leaf $\ol{ad}$ of $\ol{bc}$ such that the critical leaf $\ol{ab}$ separates
$\ol{bc}\sm {b}$ from $\ol{ad}\sm a$ inside $\disk$. However, this implies that $\ol{ab}$ is not on the boundary of $G$, a contradiction.
\end{proof}

In the next lemma we continue studying properties of all-critical sets.

\begin{lem}\label{l:no-po-ed} Suppose that an edge of an all-critical set $T$ compatible with an invariant lamination $\lam$
is the limit of a sequence of leaves of $\lam$ disjoin from $T$. Then the $\si_d$-orbit of the point
$\si_d(T)$ consists of points of the circle 
each of which is neither an endpoint of a leaf or a vertex of a gap of $\lam$.
\end{lem}


\begin{proof}
If leaves of $\lam$ approach an edge $\ell$ of $T$ then for geometric reasons their iterated
$\si_d$-images separate all points of the orbit of $\si_d(T)$ from the rest of the circle.
\end{proof}

In what follows we will rely upon the following useful tool.

\begin{dfn}\label{d:canon}
Given a periodic hyperbolic gap $W$, define a monotone map $\psi_W$ collapsing all edges of $W$ and semiconjugating the first return map
on $\bd(W)$ with $\si_r$ with appropriate $r$. Call this the \emph{canonical model} of the first return map on $\bd(W)$.
Call $\psi_W$ a \emph{canonical map} of $W$.
\end{dfn}

We use an indefinite article here because there are $d-1$ options for $\psi_W$.
Often one obtains results for $\si_m$ and ``lifts'' them to the cycle
of hyperbolic gaps that includes $W$. This is a standard argument referred to in what follows as
``lifting the model'' argument.

A chord in $W$ can be mapped by $\psi_W$
down to the disk by applying $\psi_W$ to its endpoints and then connecting the endpoints. We
can also apply $\psi_W$ to convex hulls of closed subsets of $\bd(W)$. We call $\psi_W$-images of gaps/leaves contained
in $W$ \emph{canonical}, too. Finally, $\psi_W$-pullbacks of points are called \emph{$\psi_W$-fibers}.
For $x\in \bd(W)$ let $F_W(x)\subset \bd(W)$ be the $\psi_W$-fiber containing $x$.

Observe that Example \ref{e:incompat} and the general phenomenon described after it can now be reflected in terms of canonical
models. Namely, suppose that a proper lamination $\lam$ generates a laminational  equivalence relation $\sim$ and a q-lamination $\hlam$.
Moreover, suppose that $\lam$ has a periodic hyperbolic gap $W$ which contains several critical chords from $\cpc^+$ where $\cpc$ is a
full critical collection. Then canonical images of these critical chords will be critical chords or points inside the canonical image of $W$.


\begin{lem}\label{l:gapinv}
Let $C$ be a portal continuum with $k$ vertices $x_i, i=0, 1, \dots, k-1$
and $V\supset C$ be the periodic gap of $\eYc_\cpc$ containing $C$. Let $F_i$
be the $\psi_V$-fiber containing $x_i$. Then all fibers $F_i$ are distinct
and separated in $\bd(V)$ by subarcs of $\bd(V)$ that contain infinitely many edges
(in particular, edges of portals are not leaves of $\eYc_\cpc$).
Moreover, $\psi_V(C)$ is a portal in $\psi_V(V)$ with $k$ vertices, too.
\end{lem}

\begin{proof}
Let $x_i$ and $x_{i+1}$ be consecutive vertices of $C$. We claim that there exists no chain of critical edges of $V$ connecting
$x_i$ and $x_{i+1}$. By way of contradiction suppose that $A=\bigcup \ell_i$ is such a chain. Then by Lemma \ref{l:no-po-ed} no $\ell_i$ is
the limit of a sequence of leaves from $\eYc_\cpc$. Consider the leaf $\ell_1$ with $x_i$ as an endpoint. If $\ell_1$ does not have $x_{i+1}$ as the other
endpoint then it is not from $\cpc^+$ and is not the limit of leaves of $\eYc_\cpc$. If $\ell_1=\ol{x_ix_{i+1}}$ then it is not from a proper critical set
of $\cpc^+$, nor can it be the limit of leaves from $\eYc_\cpc$, a contradiction with the assumption that there is a chain of critical edges
of $V$ joining $x_i$ and $x_{i+1}$.

Suppose there is a fiber $F$ containing $x_i$ and $x_{i+1}$. Since $\eYc_\cpc$ is proper, $F$ consists of finitely many edges. By the previous
paragraph the $\si_d$-image of $F$ is non-degenerate. Moreover by gap invariance $\si_d(F)=\si_d(V)$, a contradiction. Hence $\psi_V$-fibers of
distinct vertices of a portal continuum are distinct which implies the rest of the lemma.
\end{proof}


We state (without proof) a useful corollary of this lemma.

\begin{cor}\label{c:no-po-ed}
Let $C$ be a portal continuum from $\cpc^+$ with $k$ vertices. Let $V\supset C$ be a hyperbolic
periodic gap of $\eYc_\cpc$. Then the first return map on $V$ has the degree $d\ge k$. Moreover, $V$ may contain other critical sets all of which are portals.
\end{cor}

The properties described in Lemma \ref{l:gapinv} are rather important. In what follows we will tune $\eYc_\cpc$ in a step by step fashion.
At the end we obtain a laminational equivalence relation whose uniqueness is based upon the fact that along the way we maintain the property of
laminations established for $\eYc_\cpc$ in Lemma \ref{l:gapinv}.

\section{Root laminations exist}\label{s:rootlamiex}

We now refine $\lam^{pr}_{\cpc}$ and $\eYc_\cpc$ relying on
portals; using familiar language, one can say that we \emph{tune $\lam^{pr}_{\cpc}$ and $\eYc_\cpc$ inside their periodic hyperbolic gaps
and spread this tuning over the entire laminations through the pullbacks of these gaps}. However first we introduce a few concepts.

\begin{dfn}\label{l:root}
If $\cpc$ is a full critical collection, a
proper la\-mi\-nation $\lam$ with a laminational
equivalence $\sim_\lam=\sim$ is said to be a \emph{root lamination} (and $\sim$ is called a \emph{root laminational equivalence}) of $\cpc$
if the following holds: (1) given a proper critical set $A$ from $\cpc^+$,
all vertices of $A$ belong to the same $\sim$-class, and (2) for any cycle $\Bc$ of portals of period $n$ there exists a cycle of hyperbolic
gaps $\Uc$ of period $n$ such that sets from $\Bc$ are the only sets from $\cpc^+$ intersecting the interior of gaps from $\Uc$ and no two vertices of any set
from $\Bc$ are $\sim$-equivalent.
\end{dfn}

The rest of this section is devoted to the proof of the fact that root laminations and root laminational equivalences exist.

\begin{dfn}[Tuning]\label{d:tune}
Let $\lam_1$ and $\lam_2$ be proper laminations associated with laminational equivalences $\sim_1$ and $\sim_2$.
The lamination $\lam_2$ \emph{tunes} the lamination $\lam_1$ (denoted $\lam_1\prec \lam_2$) if and only if for any $x, y\in \uc$
with $x\sim_1 y$ we have $x\sim_2 y$. Equivalently, classes of $\sim_2$-equivalence are unions of finitely many classes of $\sim_1$-equivalence.
If $\sim_2\ne \sim_1$ we say that $\lam_2$ \emph{non-trivially} tunes $\lam_1$.
\end{dfn}

Thus, $\lam_2$ creates new classes of equivalence by uniting finite collections of $\sim_1$-classes into
$\sim_2$-classes. 
This requires drawing new leaves that establish
the new connections among $\sim_1$-classes. On the other hand, if we want $\lam_2$ to be a q-lamination, then some leaves of $\lam_1$ may find themselves inside
convex hulls of $\sim_2$-classes and may need to be erased. 
Since the concept of tuning applies not to proper laminations themselves but to the laminational equivalence relations generated
by the proper laminations, it works in the same fashion whether we talk about a q-lamination or just a proper lamination generating the same laminational
equivalence relation.

\begin{dfn}[Tuning inside gaps]\label{d:tune1}
Let $\lam_1, \lam_2$ be proper laminations where $\lam_1$ has a cycle $\Uc$ of hyperbolic gaps. Let $\lam_1\prec \lam_2$ so that
leaves from $\lam_2\sm \lam_1$ (if any) appear inside gaps of the grand orbit
of $\Uc$, but outside the grand orbit of $\Uc$ the leaves of $\lam_2$ and $\lam_1$ are the same.
Then we say that $\lam_2$ \emph{tunes $\lam_1$ inside $\Uc$}. If $\sim_1\ne \sim_2$
then we say that $\lam_2$ \emph{non-trivially} tunes $\lam_1$ inside $\Uc$.
\end{dfn}

Let us study tuning a bit more.

\begin{lem}\label{l:tuning1}
Suppose that $\lam_1\prec \lam_2$ and $\lam_2$ is a q-lamination. Then leaves of $\lam_2$ and $\lam_1$ cannot cross.
Moreover, any leaf $\ell\in \lam_2\sm \lam_1$ belongs to 
an infinite gap $U$ of $\lam_1$, and $U$ cannot be a Siegel gap.
\end{lem}

\begin{proof}
If a leaf $\ell_1\in \lam_1$ and a leaf $\ell_2\in \lam_2$ cross then it follows that all 4 endpoints of these leaves
are $\sim_2$-equivalent which contradicts the fact that $\lam_2$ is a q-lamination. To prove the second claim, observe that by the first claim
the leaf $\ell$ is a diagonal of a gap $G_1$ of $\lam_1$. If $G_1$ is finite, then it is contained in a gap $G_2$ of $\lam_2$
which is also finite (by Definition \ref{d:tune}). However then $\ell$ is a leaf of $\lam_2$ which is a diagonal of a finite gap
$G_2$ of $\lam_2$, a contradiction. Finally, observe that eventual forward images of any chord
that connects vertices of a periodic Siegel gap will cross. This completes the proof.
\end{proof}

Suppose that a q-lamination $\lam_2$ tunes a q-lamination $\lam_1$, but the containment
$\lam_2\supset \lam_1$ fails. Choose a leaf $\ell\in \lam_1\sm \lam_2$.
If $\ell$ is the limit of other leaves of $\lam_1$ then their endpoints are $\si_2$-equivalent, too, and
so $\ell\in \lam_2$. Hence
we may assume that $\ell$ is a common edge of gaps $G$ and $H$ of $\lam_1$. Assume that $H$ is infinite (one of these gaps \emph{must} be infinite).
Since $\ell\notin \lam_2$ then there must exist a finite gap $T$ of $\lam_2$ such that $\ell$
is its diagonal. Thus, $\lam_2$ has a concatenation of leaves inside $H$ that starts at one endpoint of $\ell$ and terminates at the other.
A similar picture takes place on the other side of $\ell$. This kind of dynamics is necessary and sufficient
for $\ell\in \lam_1\sm \lam_2$.

As an example take invariant cubic gaps $G$ and $H$ of $\lam_1$ with a common fixed edge $\ell$.
Then both $G$ and $H$ contain other fixed (possibly degenerate) edges $\ell_G$ and $\ell_H$.
Tune $\lam_1$ by adding to it the invariant set $\ch(\ell_G, \ell, \ell_H)$ and all its pullbacks. Clearly, $\ell\in \lam_1\sm \lam_2$.

Our tunings will satisfy a major condition: they are compatible with $\cpc^+$. As we saw before, while
$\eYc_\cpc$ is always compatible with $\cpc^+$ by construction, the q-lamination $\lam^{pr}_\cpc$ may be,
strictly speaking, incompatible with $\cpc^+$ because some ``artificially added'' leaves of $\lam^{pr}_\cpc$ may cross chords of
$\cpc^+$. However if we consider proper laminations rather than q-laminations, we will have the desired compatibility.


Around each cycle of portals $\Ac$ we construct, in a step by step fashion, a cycle of hyperbolic gaps
\emph{of the same period as $\Ac$} with interiors of these cycles of hyperbolic gaps
\emph{intersecting no other critical sets from $\cpc^+$ except for elements of $\Ac$}.
This is done by tuning
$\eYc_\cpc$ inside grand orbits of their hyperbolic cycles \emph{as much as possible}.
We apply the results of the previous sections to the first return maps
on the periodic hyperbolic gaps of $\eYc_\cpc$.

We want one more condition to be fulfilled by the proper laminations that we are constructing. Namely, we want all these
proper laminations to be \emph{degree preserving}.

\begin{dfn}\label{d:non-por}
Suppose that $\lam$ is a proper lamination with an $n$-periodic hyperbolic gap $U$
such that $\si_d^n|_{\bd(U)}$ has a portal $C$ with
two vertices of $C$ belonging to the same $\psi_U$-fiber (equivalently,
two vertices of $C$ are connected with finitely many edges of $U$).
Then $\lam$ (or $U$, or the orbit of $U$) are said to be \emph{degree reducing}.
Otherwise $\lam$ (or $U$, or the orbit of $U$) are called \emph{degree preserving.}
\end{dfn}

Here is a useful example that illustrates this concept.

\begin{ex}\label{e:portal}
Let an invariant cubic q-la\-mi\-na\-tion have a quadratic invariant gap located
between the chord $\ol{\frac13 \frac23}\subset \cdisk$ and the point $0\in \uc$.
This defines a unique q-lamination $\lam$, and $\lam$ is degree reducing. \hfill $\Box$
\end{ex}

Since $\lam$ from Definition \ref{d:non-por} is proper, no chain of edges of $U$ can connect a periodic vertex $c\in C$
and another vertex of $C$. However two non-periodic vertices $a, b$ of $C$ can be, in general, connected with a chain of
finitely many edges of $U$. In that case we can add the chord $\ol{ab}$ to $\lam$ together with all its pullbacks
consistent with $\lam$, and construct a new proper and degree reducing lamination $\lam'$. Observe, that
by Lemma \ref{l:gapinv} the proper lamination $\eYc_\cpc$ is degree preserving.

Let us discuss the concept of a degree reducing lamination. We want to apply it to proper laminations whose critical sets are
hyperbolic gaps. This is exactly the situation that we are now facing
as we are considering a cycle of a hyperbolic gap $V$ of $\eYc_{\cpc}$ and laminations that tune it while staying
compatible with portals spread around the orbit of $V$ (by our construction the only critical sets from $\cpc^+$ in
the orbit of $V$ are various portals). 

Proper laminations compatible with these portals will have hyperbolic gaps $U$ containing
the portals. 
We do not want to have among them a hyperbolic gap whose
boundary has a fiber containing two vertices of a portal.

This ensures that the degree of the first return map on $U$ is based on the cardinalities of portals
located in the gaps from the orbit of $U$;
by ruling out degree reducing laminations we guarantee that this property holds.

We work in the following setup.

\begin{peset}
Let $\tlam$ be a 
proper lamination compatible with a full critical collection $\cpc^+$.
Suppose that $\tlam$ has a degree preserving cycle $\Uc$ of hyperbolic gaps $U_0=\si_d^n(U_0),$ $\dots,$ $U_i=\si_d^i(U_0),$
$\dots,$ $U_{n-1}$; let the degree of $\si_d^n|_{\bd(U_0)}$ be $m$ and assume 
that $U_0$ contains a portal $A$ with vertex $a$.

The family of sets from $\cpc^+$ that are $\sqsubset$-contained in the gaps from $\Uc$ is denoted by $\cpc^+(\Uc)$.
Evidently, sets of $\cpc^+(\Uc)$ exhaust the criticality of $\si_d|_{\bigcup \Uc}$.
Let $X\subset U_i, X\in \cpc^+(\Uc)$; take all $\si_d^i$-preimages of vertices of $X$ that belong to $\bd(U_0)$
and call their convex hull the \emph{first pullback of $X$ to $U_0$}.
The first pullbacks to $U_0$ of sets of $\cpc^+(\Uc)$ are all-critical sets of $\si_d^n$ $\sqsubset$-contained in $U_0$
forming a full critical collection $\cpc^+(\si_d^n|_{U_0})$; these all-critical sets are said to be \emph{$\cpc^+$-induced}.
\end{peset}

We consider only critical sets that enter the interior of gaps from $\Uc$ because other
sets from $\cpc^+$ that may share edges with gaps from $\Uc$ do not impact the dynamics inside $\Uc$.

\begin{lem}\label{l:basic2}
The collection $\cpc^+(\si_d^n|_{U_0})$ consists of all-critical sets. If a chord $\ell$ is compatible with $\cpc^+(\si_d^n)$
under the action of $\si_d^n$ then $\ell$ is compatible with $\cpc^+$ under the action of $\si_d$.
If a chord $\ell$ is compatible with $\cpc^+(\si_d^n|_{U_0})$
under the action of $\si_d^n$ then $\ell$ is compatible with $\cpc^+(\Uc)$ under the action of $\si_d$.
\end{lem}

\begin{proof}
Left to the reader.
\end{proof}

Lemma \ref{l:periochord2} shows (in some cases) the existence of a non-empty
degree preserving  $\si_d^n$-invariant lamination inside $U_0$. We say that a chord
is \emph{laminational} if no two iterated forward images cross each other. 

\begin{lem}\label{l:periochord2}
Under the Periodic Setup, let $\ell\sqsubset U_0$ be
a $\si_d^n$-periodic proper chord compatible with $\cpc^+(\si_d^n|_{U_0})$.
Then there exists a proper degree preserving non-empty $\si_d^n$-invariant lamination $\lam$ inside $U_0$ compatible with
$\cpc^+(\si_d^n|_{U_0})$ and such that $\ell$ is contained in a finite gap/leaf of that lamination.
\end{lem}

\begin{proof}
First we show that
there exists a \emph{laminational} $\si_d^n$-pe\-ri\-odic chord with the same properties as $\ell=\ol{xy}$. Consider the union $L$ of all iterated forward
images of $\ell$. We claim that $L$ breaks down into several concatenations (loops) of chords that do not share vertices (however,
chords from loops can cross). This is obvious if $x$ and $y$ do not belong to the same orbit (then each loop is one chord).
Otherwise, let $m$ be the least positive integer with $\si_d^m(x)=y$. Applying $\si_d^m$ to $\ell$ over and over, we will obtain a concatenation
(a loop) of iterated $\si_d^m$-images of $\ell$ that comes back to $\ell$ after several steps. Evidently, the entire $L$ breaks down into
several loops, in a loop vertices do not repeat
until we come back, chords of loops may cross, but distinct loops cannot share vertices as desired.

We claim that each loop is contained in the closure of one component of $\disk\setminus \bigcup \cpc^+(\si_d^n|_{U_0})$.
Indeed, otherwise there are chords $\ol{ab}\subset L$ and $\ol{bc}\subset L$
that belong to the closures of $2$ distinct complementary components of $\disk\setminus \bigcup \cpc^+(\si_d^n|_{U_0})$.
Clearly, there is a portal $X$ separating $\ol{ab}$ and $\ol{bc}$ inside $\disk$. Since $\ol{ab}$ and $\ol{bc}$
evidently come from the same loop, we have a contradiction (to come back to $\ol{ab}$ our chords must
either cross $X$ or pass through its vertices, and both are impossible).

Let $A$ be a component of $L$. Then $A$ is the union of one or more loops. The properties of loops imply
that $A$ is contained in the closure of one complementary component of $\disk\setminus \bigcup \cpc^+(\si_d^n|_{U_0})$.
Set $B=\ch(A)$.
Evidently, $B$ is a finite polygon with periodic vertices, all of the same period.
Clearly, any edge of $B$ is a periodic \emph{laminational} chord compatible with $\cpc^+(\si_d^n|_{U_0})$.
Hence from now on we assume that $\ell=\ell_0$ is a laminational edge of a finite (possibly, degenerate) polygon $B_0\supset  \ell$
with pairwise disjoint images and of period is, say, $k$. The boundary of $B_0$ is one of the loops from above.

Clearly, 
$B_0$ is
contained (except for endpoints) in one component $X_0$ of $\disk\setminus \bigcup \cpc^+(\si_d^n|_{U_0})$.
We will inductively construct full $\si_d^n$-sibling collection for $B_i$ and its pre-images.
We will do this by first constructing \lq\lq partial siblings\rq\rq \, in the closure $X$ of each
complementary component of $\disk\setminus \bigcup \cpc^+(\si_d^n|_{U_0})$. We have three cases depending on the
location of $X$ with respect to $B_0$.  Set $B_i=\si_d^{ni}(B_0), 0\le i\le k-1$.


Case 1.  $X=X_0$. Then $B_0$ itself is the partial sibling of $B_0$ contained in $X_0$.

Case 2. $X$ contains a non-periodic vertex $x$ of a portal $P$ whose periodic vertex  $p$ is in $B_0$ and
$X$ is located on the opposite side of $\ol{px}$ from $B_0$. It is easy to see
that there is a unique edge $\ol{xy}$ of $P$ contained in the boundary of $X$. Define the partial sibling $B_X$ of $B_0$ in $X$
as follows: it is the convex hull of all sibling points of vertices of $B_0$ in $X$ except $y$.  Note that $B_0$ and $B_X$ are disjoint.

Case 3. For all remaining complementary components take the convex hull of all sibling points of vertices of $B_0$ in
$X$.

Consider the sibling sets of $B_i$ for some $B_i$. These sets have the property that the map $\si_d^n$ restricted on a sibling set
$B_X$ of $B_i$ maps its boundary forward onto $B_{i+1}$ in the non-strict order preserving fashion.
Suppose that a sibling $B_X$ of $B_i$ intersects a proper critical set $Z\in \cpc^+(\si_d^n|_{U_0})$. Then by definition $B_X$
contains an edge of $Z$. Since $Z$ is all-critical, all its edges are also edges of other siblings of $B_i$ contained in various
complementary components. The convex hull of the union of all sibling sets of $B_i$ that share an edge with $Z$ is a convex set that
has the desired property. We can continue enlarging the just obtained set if it intersects another proper critical set from $\cpc^+(\si_d^n|_{U_0})$
repeating the arguments almost verbatim.

On the other hand, if proper critical sets from $\cpc^+(\si_d^n|_{U_0})$ are disjoint
from the currently constructed set, say, $C$, then it follows that $C$ is disjoint from the remaining sibling sets of $B_i$. Indeed,
by definition distinct sibling sets of $B_i$ can only intersect if they share a vertex, and by our construction
this is only possible in the situation described above, a contradiction.

Each sibling $C^i_j$ of $B_i$
maps onto $B_{i+1 \,\mod\, k}$ and all $C^i_j$ are pairwise disjoint.  Hence by Lemma~\ref{l:sipul} each such set can be
pulled back in a straightforward way to a full sibling collection. Doing this inductively yields a proper invariant prelamination $\tlam$.
By Theorem \ref{t:laclo} and Lemma \ref{l:prop-clos} the closure $\lam$ of $\tlam$ is a proper invariant lamination
that respects the family of sets $\cpc^+(\si_d^n|_{U_0})$.

Finally suppose that an edge $\tell$ of a portal $A$ is an element
of $\lam$. Assume that the periodic vertex $v$ of $A$ is of period $N$. Since $\lam$ is proper $\tell$ has non-periodic endpoints. By our construction $\tell\notin \tlam$.
Hence $\tell$ is the limit of a sequence $\ell_i$ of leaves of $\tlam$. Since $\tlam$ is proper, at most finitely many leaves $\ell_i$ may share
a vertex with $\tell$. It follows that leaves $\si_d^N(\ell_i)$ cross $A$, a contradiction.
\end{proof}


Lemma \ref{l:periochord2} shows that to refine laminations we need to find periodic chords compatible with $\cpc^+$.


\begin{lem}\label{l:periochord}
If 
$\cpc^+$ includes more than one portal for some $n>0$ then there exists a periodic chord compatible with $\cpc^+$ and
$\sqsubset$-contained in $U_0$. 
Thus, if $\cpc^+$ includes more than one portal, or if there are portals of period greater than $1$ then
there exists a periodic chord compatible with $\cpc^+$. 
\end{lem}

\begin{proof}
First we prove the claim in the case when $\si_d$ itself includes two portals with fixed vertices.
Let $a\ne b$ be fixed points
such that $a$ is a vertex of a portal $\Cc_a\in \cpc^+$, while $b\ne a$ is a vertex of a portal $\Cc_b\ne \Cc_a$.
We claim that then there exists an invariant chord with $\si_d^n$-fixed endpoints compatible with $\cpc^+$.

We use a counting argument. We need to show that at least one complementary component contains at least two fixed points in its closure.
There are exactly $d$ complementary components
$W_1,$ $\dots,$ $W_d$ to the union of polygons from $\cpc^+$.
Let us associate with each complementary component $W_i$ the set $F_i$ of fixed points that belong to $\bd(W_i)$.
By way of contradiction assume that all sets $F_i$ consist of at most one point. Then
by the assumptions it follows that $\{a\}$ coincides with two sets $F_i$ (say, $F_1$ and $F_2$),
and $b$ coincides with two other sets $F_i$ (say, $F_3$ and $F_4$).
After that there remain $d-4$ sets $F_5,$ $\dots,$ $F_d$ and $d-3$ fixed points (overall there are $d-1$ fixed points of $\si_d$). It follows that
there exists at least one set $F_j$ consisting of at least two points.

Now, choose $N$ such that $\si_d^N$ fixes all periodic vertices of our portals. By the preceding paragraph there is a $\si_d^N$-invariant chord $\ell$ compatible with
$\cpc^+(\si_d^N)$. By Lemma \ref{l:basic2} it is compatible with $\cpc^+$. The rest of the lemma easily follows.
\end{proof}

The next corollary of Lemma \ref{l:periochord} is immediate. Recall (Definition \ref{d:canon}) that $\psi_{U_0}$ is the canonical
semi-conjugacy between $\si_d^n|_{\bd(U_0)}$ and $\si_m$ (we assume the  Periodic Setup).

\begin{cor}\label{c:periochord1}
If $\cpc^+(\si_d^n|_{U_0})$ has two portals whose periodic vertices have distinct $\psi_{\bd(U_0)}$-images
then there is a periodic chord $\ell\sqsubset U_0$ compatible with $\cpc^+$.
In particular, if there are portals of greater than $n$ periods $\sqsubset$-contained in $U_0$ then
there is a periodic chord $\ell\sqsubset U_0$ compatible with $\cpc^+$.
\end{cor}

\begin{proof} By the assumptions, the map $\psi_{U_0}$ induces two distinct portals of $\si_m$. By Lemma \ref{l:periochord} this implies the existence of
a periodic chord compatible with the $\psi_{U_0}$-images of portals. Lifting it back to $U_0$ we will find the desired chord $\ell$. The second
claim now follows. The details are left to the reader.
\end{proof}

Let us now describe the following situation.
Assume that a proper lamination $\lam$ has an invariant gap $U$ with a non-degenerate edge $\ol{xy}$ where $\si_d(x)=x, \si_d(y)=y$.
Assume that the degree of $\si_d|_{\bd(U)}$ is 3 and that $\cpc$ includes critical leaves $\ol{xa}$ and $\ol{yb}$ such that
their other endpoints $a$ and $b$ are not identified by $\psi_U$.
Then the model of $\si_d|_{\bd(U)}$ induced by the map $\psi_U$ will include the critical leaves $\ol{0\frac13}$ and $\ol{0\frac23}$.
In the model these two leaves form one portal; by construction we will complete them with the leaf $\ol{\frac13 \frac23}$
to form a portal, say, $A$.
Clearly, there are no periodic chords compatible with $A$. Hence
the arguments similar to those in Corollary \ref{c:periochord1}
(when we discover some phenomenon in $\psi_U(U)$, i.e ``downstairs'', and lift it to $U$, i.e. ``upstairs'') will not work.
This motivates the next lemma.

\begin{lem}\label{l:portal1}
Under the Periodic Setup, assume that
the $\psi_{U_0}$-images of all sets from $\cpc^+(\si_d^n|_{U_0})$ share
a $\si_m$-fixed vertex $0$, but there are at least two portals in
$\cpc^+(\si_d^n|_{U_0})$. Then
a periodic chord $\ell\sqsubset U_0$ compatible with $\cpc^+$ which
separates portals in $U_0$ into two non-empty groups of portals exists.
\end{lem}

\begin{proof}
Take a $\si_d^n$-portal $A\subset U_0$ with a periodic vertex $a$ and a $\si_d^n$-portal $B\subset U_0$ with vertex $b$ such that $A\ne B$ and
there are no critical sets separating $A$ from $B$ in $\lam$. Since by the assumptions $\psi_{U_0}(a)=\psi_{U_0}(b)=0$ then $a$ and $b$ belong to
the same fiber $\psi_{U_0}^{-1}(0)=X$ which is a finite chain of $\si_d^n$-invariant edges of $U_0$.
Choose edges $\oa$ of $A$ and $\ob$ of $B$ that are not separated in $\disk$
by other edges of $A$ and $B$. We claim that the other (not equal to $a$ or $b$) vertices
$a'$ and $b'$ of these edges cannot be vertices of the same sibling set $X'$ of $X$. Indeed,
going from $a'$ to $b'$ along $X'$ is not done the same direction as from $a$ to $b$ along $X$,
which contradicts $X$ and $X'$ being siblings.



This implies that $\psi_{U_0}(a')\ne \psi_{U_0}(b')$ are two distinct points mapped by $\si_m$ to the same $\si_m$-fixed point
$\psi_{U_0}(a)$. Hence there exists a $\si_m$-fixed point $z$ that belongs to the circle arc between $\psi_{U_0}(a')$ and $\psi_{U_0}(b')$.
The $\psi_{U_0}$-preimage of $z$ is a $\si_d^n$-invariant (possibly degenerate) edge of $U_0$. Connecting one of its endpoints to $a$ or $b$ we will discover
a desired chord.
\end{proof}

Let us now state a principle that helps one to construct invariant laminations.

\begin{spr} Suppose that under the Periodic Setup one has constructed a non-empty $\si_d^n$-invariant proper lamination in $U_0$. Then
by 
pulling it back one can spread it over
the grand orbit of $U_0$ (or, equivalently, $\Uc$). This collection of leaves together with $\tlam$ forms a new proper lamination that non-trivially tunes
$\tlam$ inside $\Uc$.
\end{spr}

We leave the verification of this principle to the reader.

\begin{dfn}\label{d:funda}
Under the Periodic Setup, call a cycle $\Uc$ of hyperbolic gaps with portals \emph{fundamental}
if and only if $\si_d^n|_{\bd(U_0)}$ has a unique critical set which is a portal $A$ with a $\si_d^n$-fixed vertex
and $\lam$ is degree preserving.
\end{dfn}

The next lemma is straightforward.

\begin{lem}\label{l:l:funda}
Under the  Periodic Setup, a cycle $\Uc$ of hyperbolic gaps is fundamental if and only if there is
a cycle $\Ac$ of portals inside $\Uc$, it is of period $n$, and for every $i$ such that
$\si_d|_{\bd(U_i)}$ has degree $k>1$ there is a portal $A_i\in \Ac, A_i\subset U_i$ with $k$ vertices.
\end{lem}

\begin{proof} Left to the reader.
\end{proof}

The next lemma studies fundamental cycles of gaps.

\begin{lem}\label{l:funda1}
Under the Periodic Setup if $\Uc$ is fundamental then there exists no
proper degree preserving lamination non-trivially tuning $\tlam$ inside
$\Uc$.
\end{lem}

In other words, any proper lamination $\lam\supset \tlam$ such that $\lam\sm \tlam$ is contained in the grand orbit of $\Uc$ and is
compatible with all portals gives rise to the same laminational equivalence relation as $\tlam$, i.e. to $\sim_\tlam$.

\begin{proof}
Suppose that $\Uc$ is fundamental.
Let $A\sqsubset U_0$ be a portal with $\si^n_d$-fixed vertex $a$.
Suppose there is a non-trivial degree preserving tuning $\lam$ of $\tlam$ inside $\Uc$ despite the fact that $\Uc$ is fundamental.
We claim that the canonical map $\psi_{U_0}=\psi$ (see Definition \ref{d:canon}) induces a non-empty proper $\si_m$-invariant
lamination $\lam'$ based on the restriction of $\lam$ on $U_0$
(evidently, this restriction is $\si_d^n|_{\bd(U_0)}$-invariant). Recall that according to the Periodic Setup $m$ is the degree of $\si_d^n|_{U_0}$.

All claimed properties of $\lam'$ except for being degree preserving are immediate. Since $\lam$ is
degree preserving, it has a gap $W\supset A$ such that the degree of $\si_d^n|_W$ is $m$. Hence $\lam'$ has a gap
$\psi(W)$ of degree $m$ implying that $\psi(W)=\cdisk$. Thus, leaves of $\lam$ connect points of
the same fiber which implies that $\lam$ gives rise to the same laminational equivalence relation as $\tlam$.
\end{proof}

Finally, we can prove the main result of this section.

\begin{thm}\label{t:rootexist}
For every full critical collection $\cpc$ there exists a root lamination $\lam$.
\end{thm}

\begin{proof}
Consider $\eYc_\cpc$. Recall that by Theorem \ref{t:side-lam} $\eYc_\cpc$ is a proper lamination.
Choose a cycle $\Uc$ of hyperbolic gaps of $\eYc_\cpc$. Evidently, the Periodic Setup applies to $\Uc$ with $\tlam=\eYc_\cpc$.
We claim that if $\Uc$ is not fundamental then there exists a non-trivial degree preserving tuning of $\tlam$ inside $\Uc$.

Let us prove that if $\Uc$ is not fundamental then there is a periodic chord compatible with $\cpc^+$ and $\sqsubset$-contained in
$U_0$. To this end we classify $\si_d^n$-critical sets induced by $\cpc^+$ depending on how
portals are distributed over $\Uc$. By the Periodic Setup, there is a portal $A\subset U_0$ with a periodic vertex $a$.

Suppose that there is a portal $B\subset U_i, i\ne 0,$ with vertex $b\ne \si_d^i(a)$.
It follows that the pullback of $b$ to a point  in $\si_d^{n-i}(b)$ (which maps by $\si_d^i$ to $b$) and the simultaneous pullback of
$B$ yield a portal $B'\ne A$ of $\si^n_d|_{\bd(U_0)}$.
Moreover, suppose that $a$ is of period $mn$ with $m>1$. Then by pulling
$a$ itself back $n$ steps under $\si_d$ to $\si_d^{n(m-1)}(a)$ we discover a portal $B'\ne A$ of $\si^{mn}_d|_{\bd(U_0)}$. Finally, it may happen
that even without additional pullbacks and constructions there is a portal $B\subset U_0, B\ne A$.

In all three cases  by Corollary \ref{c:periochord1} and Lemma \ref{l:portal1}
there is a periodic chord in $U_0$ compatible with $\cpc^+(\si_d^n|_{U_0})$.
By Lemma \ref{l:periochord2} this implies the existence of a proper degree preserving non-empty
$\si_d^n$-invariant lamination inside $U_0$ compatible with $\cpc^+(\si_d^n|_{U_0})$. By pulling it back within the grand orbit
of $U_0$ we obtain the desired tuning of $\tlam$.

We showed that if $\Uc$ is not fundamental then there is a non-trivial degree preserving tuning of $\tlam$ inside $\Uc$.
The new lamination is compatible with $\cpc^+$ and degree preserving. Hence the arguments can be repeated. This yields a sequence
of proper degree preserving laminations $\lam_0=\eYc_\cpc,$ $\lam_1, \dots$ where $\lam_{j+1}$ non-trivially tunes
$\lam_j$ inside a non-fundamental cycle $\Uc_j$ of hyperbolic gaps of $\lam_j$. We claim that this sequence is finite.

Indeed, all critical sets of $\cpc^+$ in hyperbolic gaps of $\eYc_\cpc$ are portals. Hence each lamination $\lam_j$ has a few
periodic critical hyperbolic gaps containing portals from $\cpc^+$, as well as finite critical sets of $\eYc_\cpc$, and this exhausts the list
of critical sets of $\lam_j$. If the distribution of portals among cycles of hyperbolic gaps remains the same as we move from $\lam_j$ to
$\lam_{j+1}$, and if the periods of cycles also remain the same, then it is easy to see that $\lam_{j+1}$ will have to coincide with
$\lam_j$. Indeed, otherwise we will have a cycle of hyperbolic gaps $\Uc$ of $\lam_j$ and a cycle of hyperbolic gaps $\Vc\subsetneqq \Uc$ of the same
period and containing all portals that are contained in $\Uc$ (i.e., each element of $\Vc$ is a proper subset of some element of $\Uc$)
which is clearly impossible (an invariant gap strictly contained inside another invariant gap has to be of smaller degree).

By definition the last proper lamination in the constructed sequence of laminations is a root lamination as desired.
\end{proof}

\section{Root laminations are unique}\label{s:root1}

Now we will prove the uniqueness of a root lamination.
Consider an example.

\begin{ex}\label{e:sim}
Consider $\si_m$ and the convex hull $\Si_m$ of $0$ and all its $\si_m$-preimages. Clearly,
$\Si_m$ is an $m$-gon and a portal. Pull $\Si_m$ back to each partially critical component of $\cdisk\sm \Si_m$,
i.e. to the parts of the unit disk enclosed by the edges of $\Si_m$ and $m$ circular arcs $[\frac{i}m, \frac{i+1}m], 0\le i<m$
called \emph{arcs of the first generation}.
Each such pullback of $\Sigma_m$ into one of these arcs is an $m+1$-gon with an
edge of $\Si_m$ serving as one of its edges; these form a ``ring'' of circularly ordered $m+1$-gons, consecutively sharing vertices
of the form $\frac{j}m$. We say that these pullbacks of $\Si_m$ are \emph{of the first generation}.

The unit circle $\uc$ is partitioned by the vertices of the pullbacks of $\Si_m$ of the first generation into
the circle arcs of the form $[\frac{i}{m^2}, \frac{i+1}{m^2}]$ called \emph{arcs of the second generation}.
Each such arc $I$ maps onto an arc of the first generation and, hence, covers the vertices of a pullback of the first generation of $\Si_m$, say,
$B$. Hence $B$ can be pulled back to the convex hull of $I$ forming a new $m+1$-gon. This can be done with every arc of the second generation.
Then this process is repeated.

The unions of 
pullbacks of $\Si_m$ of generations up to $i, i\to \infty,$ form a growing sequence of gaps
that fill up from within the open unit disk $\disk$. These gaps ``touch'' $\uc$ at their vertices that are
points of the form $\frac{i}{m^k}$ in lowest terms
where $k$ is the generation of the 
pullback of $\Si_m$ containing the point.
Evidently, in the limit these gaps fill up the open unit disk $\disk$ while also touching $\uc$ at points $\frac{i}{m^k}$. \hfill $\square$
\end{ex}


Example \ref{e:sim} will be used in the proof of Theorem \ref{t:rootuni}. 

\begin{dfn}\label{d:circ-poly} A closed convex set $A\subset \cdisk$ whose boundary $\bd(A)$ is the union of alternating
 arcs of $\uc$ and
chords is called a \emph{circular polygon}. Consider the image of $\bd(A)$ under the map
that coincides with $\si_d$ on the arcs and sends a chord $\ol{ab}$ to the chord $\ol{\si_d(a) \si_d(b)}$ affinely;
abusing the notation we will denote this
map by $\si_d$. If $\si_d(\bd(A))$ is a Jordan curve and $\si_d|_{\bd(A)}$ is
one-to-one except for (possibly) some critical chords in $\bd(A)$ we call $A$ a \emph{monotone} circular polygon. 
Abusing the notation we set $\ch(\si_d(\bd(A)))=\si_d(A)$; clearly, $\si_d(A)$ is a circular polygon.
\end{dfn}

The next definition develops the above further and makes it more dynamic.

\begin{dfn}\label{d:dynam1} The non-degenerate convex hull $\ch(X)$ of a closed (finite) set $X\subset \uc$ is a \emph{(finite) gap-leaf}.
If a finite gap-leaf $B$ is contained in $\si_d(A)$ where $A$ is a monotone circular polygon, then the unique maximal by
inclusion gap-leaf $B'\subset A$ that maps onto $B$ by $\si_d$ is called the \emph{pullback of $B$ into $A$}. Abusing the notation,
denote $B'$ by $\si^{-1}_d(B)$. Observe that these pullbacks are convex.
\end{dfn}

In Definition \ref{d:dynam1}, if no vertex of $B$ is the image of a critical edge of $A$ then $\si^{-1}_d(B)$ and $B$ have the same
number of vertices. Otherwise each vertex $x$ of $B$ which is the image of a critical edge $\ol{x}$ of $A$ is pulled back
to $\ol{x}$ or to a concatenation of $\ol{x}$ and other critical edges; in this case the number of vertices of $\si^{-1}_d(B)$ is greater.

Now, given a full critical collection $\cpc$, construct $\cpc^+$.
Components of $\cdisk\sm \bigcup \cpc^+$
are monotone circular polygons, and for each such component $A$ we have $\si_d(A)=\cdisk$. Consider
pullbacks of sets of $\cpc^+$
into the components of $\cdisk\sm \bigcup \cpc^+$. Since sets of $\cpc^+$ are pairwise disjoint,
their
pullbacks are pairwise disjoint, too. 

To make the next step, observe that complementary components to the union of sets of $\cpc^+$ and their 
pullbacks
are circular polygons. Their boundaries map to the boundaries of complementary components of $\bigcup \cpc^+$ under $\si_d$ extended according
to Definition \ref{d:dynam1}. This allows us to define 
pullbacks of sets of $\cpc^+$ of the next generation.

We continue this process inductively. By induction, on the step $n$ we will have constructed all
$\si_d^n$-pullbacks of sets of $\cpc^+$. This creates complementary components to their union that will
be called \emph{complementary components of generation $n$}.
Each such component $A$ is a circular polygon with the boundary that maps to the boundary of a complementary component $X$ of generation $n-1$.
The set $X$ may contain some 
pullbacks of sets of $\cpc^+$ of generation $n$. Then their 
pullbacks into $A$ form 
pullbacks of
sets of $\cpc^+$ of generation $n+1$. Once this is done for all complementary components of generation $n$, the step of induction is made, and we can
continue.

\begin{dfn}\label{d:tile}
The iterated 
pullbacks of sets of $\cpc^+$ are called the \emph{tiles (generated by $\cpc$)}.
The family of tiles generated by $\cpc$ is called the \emph{partial tiling (generated by $\cpc$)}.
An iterated pullback of a proper set from $\cpc^+$ will be called a \emph{proper tile}.
\end{dfn}


We are ready to prove the main theorem of this section.

\begin{thm}\label{t:rootuni} Given a full critical collection $\cpc$ there exists a unique root lamination of $\cpc$.
\end{thm}

\begin{proof}
The existence of a root lamination is proven in Theorem \ref{t:rootexist}. Let us now prove that a root lamination is unique.
Let $U$ be an $n$-periodic
gap of a root lamination for $\si_d$. Then $\si^n_d|_U$ is canonically semiconjugate to $\si_m$ with the appropriately chosen $m$.
Take tiles contained in gaps from the orbit of $U$ and pulled back to $U$ under the smallest power of $\si_d$ that
brings them to $U$. Their union is a convex polygon, say, $X$, collapsing to its unique periodic vertex under $\si_d^n$. Its image under the canonical
semiconjugacy is the set $\Si_m$ (see Example \ref{e:sim}).

By Example \ref{e:sim}, in the case of $\Si_m$ the union of tiles is the entire open unit disk
$\disk$ plus all the points of $\uc$ of the form $\frac{i}{m^k}$ in lowest terms. Other points of $\uc$ do not belong to this union.
Lifting it under the canonical semiconjugacy we see that the union $Y_X$ of all pullbacks of $X$ inside $U$ covers the interior of
$U$ and the points/leaves of $\bd(U)$ that canonically map to points $\frac{i}{m^k}$ in lowest terms (we choose a canonical semiconjugacy
for $\si_d^n|_U$ so that the periodic vertex of $X$ maps to $0$).

The above arguments are implemented in the framework of a root lamination. However, $Y_X$
(which equals $U$ with some points and edges of its boundary) can be defined independently. Indeed,
let $A\in \cpc^+$ (for example, $A$ may be a portal). Consider the set $\bigcup \cpc^+\sm A$ (i.e., remove $A$ from $\cpc^+$).
Let $Z_A$ be the component of $\cdisk\sm (\bigcup \cpc^+\sm A)$ containing $A$. Then $Y_X$ is the maximal convex connected union of tiles contained in $Z_X$.
It follows that the periodic hyperbolic gap of any root lamination containing $X$ is the closure of the set $Y_X$. Thus, cycles of
periodic hyperbolic gaps containing portals coincide for all root laminations generated by $\cpc$. The iterated pullbacks of such hyperbolic
gaps depend only on $\cpc^+$, and are, therefore, well defined and independent of the choice of a root lamination.

Next we argue as follows. Let $\lam_1$ and $\lam_2$ be two root laminations of a full critical collection $\cpc$. By the above
their periodic hyperbolic gaps coincide, together with their grand orbits. Insert in those periodic gaps proper critical leaves
and do it the same way for both $\lam_1$ and $\lam_2$. Replace portals with these newly inserted leaves to create a new full critical collection
$\cpc'$. It is a Side critical collection, hence by Theorem \ref{t:side-full} the corresponding root lamination $\lam^{pr}_{\cpc'}$ is the same for
both $\lam_1$ and $\lam_2$. Since the grand orbits of periodic hyperbolic gaps coincide for both $\lam_1$ and $\lam_2$
then in fact $\lam_1=\lam_2$.
\end{proof}

\section{Concluding remarks}

In this paper, we are addressing the issue of relating critical portraits on the one hand and q-laminations/topological polynomials  on the other.
In the process, we establish equivalence between some critical portraits. Indeed, if $\lam$ is a Side non-capture q-lamination then we associate
with it any critical portrait compatible with $\lam$. As a result, all critical portraits compatible with $\lam$ are identified.

The association between critical portraits and laminations in the case when some critical leaves have periodic endpoints is
more complicated. Indeed, consider the quadratic case. In that case it is well-known that if a hyperbolic lamination $\lam$ has a critical periodic
hyperbolic gap $U$ of period $N$ then $U$ has a unique (non-degenerate) edge, say, $\ol{ab}$ of period $N$. In that case one associates to
$\lam$ a pair of critical leaves with endpoints $a$ and $b$. In other words, in the hyperbolic case the association between laminations
and critical leaves is made on the basis of critical leaves contained in the critical hyperbolic gap of $\lam$ and having an endpoint of the same period as this gap.
This motivates the next definition which repeats Definition \ref{d:introlegal} and is placed here for the sake of convenience.

\begin{dfn}\label{d:legal} Given a q-lamination $\lam$ a critical portrait $\Po$ is
\emph{legal} for $\lam$ (and vice versa) if:
\begin{enumerate}
\item no chord in $\Po$ crosses any leaf of $\lam$ inside the unit disk,
\item  for every  critical leaf with a periodic endpoint  $x=x_0$ of period $n$, there exist a cycle of hyperbolic gaps $U_0,\ldots,U_{n-1}$, also of period $n$ so that
if $\si_d|_{\bd(U_i)}$ has degree $d_i$, then there exist $d_i-1$ critical chords  of $\Po$ contained in $U_i$ whose convex hull is a portal with vertex $x_i=\si_d^i(x_0)$.
\end{enumerate}
We also say that $(\Po, \lam)$ is a \emph{legal pair}.
If two legal pairs share either a critical portrait or a lamination, we declare them
\emph{equivalent} and, as usual, spread this equivalence by transitivity.
\end{dfn}

Evidently, the pair $(\Po, \lam^{root}_\Po)$ is legal. There may exist other legal pairs equivalent to it.
By definition, switching from one legal pair to another may either (1) involve the same
lamination but distinct critical portrait, or (2) the same critical portrait but distinct lamination (the latter possibility relies upon the fact that
in Definition \ref{d:legal} we consider both degree preserving laminations and laminations that \emph{are not} degree preserving).
We refer to these transitions as \emph{type (1) and type (2) equivalences}, respectively.
Say that a critical portrait $\Pc$ is \emph{non-capture} if it is compatible with no capture q-lamination.

\begin{lem}
  \label{l:noncap-equiv}
Suppose that $\Po$ is non-capture, 
and that $(\tlam, \tPo)$ is equivalent to $(\lam^{pr}_\Po, \Po)$.
Then $\tlam\supset \lam^{pr}_\Po=\lam^{pr}_{\tpo}$.
\end{lem}

\begin{proof}
We use induction. The base of induction is immediate.
Suppose that $(\tlam, \tPo)$ is equivalent to $(\lam^{pr}_\Po, \Po)$ and has the
properties claimed in the lemma: $\tlam\supset \lam^{pr}_\Po$ and $\lam^{pr}_{\tpo}=\lam^{pr}_\Po$.
Consider the next legal pair $(\tlam', \tPo')$ equivalent to $(\tlam, \tPo)$.

(1) We have $\tlam=\tlam'$ (i.e., it is the critical portrait that varies). Then the new critical portrait $\tPo$ still has the same number of critical leaves in
the same finite critical gaps/leaves of $\lam^{pr}_\Po$ as $\tlam$ did. It follows that if we construct the preroot lamination based upon
$\tPo'$ we will obtain the same preroot lamination as before, i.e., by induction, $\lam^{pr}_{\tpo'}=\lam^{pr}_\Po$.
Since by induction $\tlam'=\tlam\supset \lam^{pr}_\Po$, we are done in this case.

(2) We have that $\tPo'=\tPo$ (i.e., it is the lamination that varies).
Then by construction it follows that the only changes in lamination that are possible may take place within hyperbolic gaps of $\lam^{pr}_\Po$ and their pullbacks.
This implies that the preroot lamination does not change and, therefore, $\tlam'\supset \lam^{pr}_\Po=\lam^{pr}_{\tpo'}$  as desired.
\end{proof}

Moreover, the definition of legal implies that the periods of critical leaves with periodic endpoints do not change when we move from one legal pair
to an equivalent one (even though the actual critical leaves may change). Since the only other critical leaves our laminations are allowed to have
inside hyperbolic gaps of $\hlam$ are contained in the corresponding portals, we conclude that there are finitely many possible parts of
critical portraits that consist of critical leaves not contained in finite gaps/leaves of $\hlam$. This finiteness implies the conclusion of the following theorem.

\begin{thm}\label{t:conclu}
Consider the family of all legal pairs $(\lam, \Po)$ such that $\Po$ is incompatible with capture laminations,
with the equivalence defined in Definition \ref{d:legal}.
Then each class of equivalence is finite.
\end{thm}

In the forthcoming papers we plan to use root lamination, legal pairs laminations and critical portrait, in order to
construct a model space of topological polynomials of degree $d\ge 3$.
\appendix

\section{Periodic gaps of degree one}\label{ss:deg1}

\setcounter{thm}{0}

We start with the following definition.

\begin{dfn}[Stand alone gaps]\label{d:stand}
Let $U'\subset \uc$ be a closed set. Suppose that $U',$ $\si_d(U'),$ $\dots,$ $\si_d^{n-1}(U')$, $\si_d^n(U')=U'$
have convex hulls denoted by $U_i$ with disjoint interiors and, for each $i$, the map from $\bd(U_i)$ to $\bd(U_{i+1})$
is the composition of a positively oriented monotone map and a positively oriented covering map.
Then $U_0$ is called a \emph{stand alone} $n$-periodic gap.
\end{dfn}

By a \emph{stand alone invariant gap}, we mean a stand alone 1-periodic gap.
Note that a stand alone $\si_d^n$-invariant gap is \emph{not} the same as a stand alone $n$-periodic gap under $\si_d$.

Let $U'\subset \uc$ be closed, let $U=\ch(U')$ be a stand alone gap of period $n$. If $x\in \bd(U)$ is such that $\si_d^{kn}(x)=x$ then either
$x\in U'$, or, otherwise, $x$ belongs to an edge $\ell=\ol{ab}$ of $U$ such that $\si_d^{nk}(a)=a$ and $\si_d^{nk}(b)=b$.
Hence, without loss of generality, from now on we always mean that periodic points $z\in \bd(U)$ belong to
$\uc$.

The following simple result is Lemma 2.16 of \cite{bopt10}.

\begin{lem}\label{l:edges-of-gaps}
Suppose that $U$ is a periodic stand alone gap. Then all its edges are either (pre)periodic or (pre)critical.
\end{lem}

A typical example of a stand alone $k$-periodic gap is described in the next folklore lemma;
we only sketch the proof leaving details to the reader.

\begin{lem}\label{l:example1}
Suppose that $\ell_1,$ $\dots,$ $\ell_m$ are pairwise unlinked critical chords that are all contained
in the boundary of one component $U$ of $\cdisk\sm \bigcup \ell_i$ where $\bd(U)$ is the union of
$\bigcup \ell_i$ and a few circle arcs. Suppose that for some $k>0$ the $\si_d^k$-image of the union of all the circle arcs from $\bd(U)$
covers $\uc$ in a $t$-to-$1$ fashion (except for images of $\ell_i$'s that will have more preimages).
Consider the set $A\subset \uc$ consisting of points whose forward $\si_d^k$-orbits are contained in
$\ol{U}$. Then the following holds.

\begin{enumerate}

\item $\ch(A)$ is a stand alone $\si_d^k$-invariant gap of degree $m$ unless $A$ consists of at most 2 points.

\item Suppose that $t=1$. Then all periodic points (if any) that belong to $A$ have the same period.

\end{enumerate}

\end{lem}

\begin{proof}
(1) Since the chords on the boundary of $U$ are critical, the $\si_d^k$-images of circle arcs on $\bd(U)$ are
concatenated, and locally the map $\si_d^k$ preserves orientation on $\bd(\ch(A))$. Since
edges of $\ch(A)$ map onto edges of $\ch(A)$, we are done.

(2) Since $t=1$, the orientation on $A$ is preserved by $\si_d$. Hence periodic points in $A$
are all of the same period.
\end{proof}

We will need the following lemma.

\begin{lem}\cite[Lemma 3.10]{botw23}\label{l:crit-must}
Suppose that $G$ is a degree one $k$-periodic infinite gap of a $\si_d$-invariant lamination $\lam$ for some $d\ge 2$.
Then some gaps from the orbit of $G$ have critical edges. Moreover,
there are two possibilities.
\begin{enumerate}
\item There is a monotone semi-conjugacy between $\si_d^k|_{\bd(G)}$ and an irrational rotation
 of\, $\uc$ that collapses all edges of $G$ to points;
 moreover, if there are concatenations of edges of $G$, then each concatenation
 consists of at most $d-1$ leaves.
\item There are periodic edges of $G$; for some minimal $q$ all periodic edges of $G$ are
$\si_d^{kq}$-fixed. Moreover, each arc $I\subset \bd(G)$ located between two adjacent $\si_d^{kq}$-fixed
edges of $G$ has the following properties:

\begin{enumerate}

\item $I$ is $\si_d^{kq}$-invariant;

\item at exactly one endpoint, say $x$,
of $I$, a $\si_d^{kq}$-critical edge $\ell\subset I$ is located (so that there are critical leaves with periodic endpoints);

\item all points of $I$ map towards
$x$ by $\si_d^{kq}$.

\end{enumerate}

\end{enumerate}

Also, $\lam$ has isolated leaves with
both endpoints non-pre\-pe\-ri\-odic, or with one periodic and one non-periodic endpoints. In particular,
$(a)$ the lamination $\lam$ is not perfect, and $(b)$ it cannot have a dense subset of (pre)periodic leaves whose endpoints
have equal preperiods.
\end{lem}

The following definition is based upon Lemma \ref{l:crit-must}.

\begin{dfn}\label{d:deg1}
Periodic gaps of degree one that have no periodic vertices are said to be \emph{Siegel gaps}. Infinite periodic gaps of degree one
that have periodic vertices are said to be \emph{caterpillar} gaps. Their non-periodic pullbacks are called \emph{pre-Siegel} and
\emph{pre-caterpillar} gaps, respectively.
\end{dfn}

Evidently, Siegel gaps are 
 uncountable.

\bibliographystyle{amsalpha}

\end{document}